\documentclass[reqno]{amsart}

\usepackage{amsfonts, upgreek}
\usepackage{amssymb, amsopn, latexsym,amsmath,graphics,verbatim,amsthm,enumerate,amscd,tabularx}
\usepackage[all,cmtip]{xy}
\usepackage{graphicx}
\usepackage{fancyhdr}
\usepackage[active]{srcltx}
\usepackage[british]{babel}
\usepackage{amsmath,amssymb,amsthm,mathrsfs, url}

\theoremstyle{plain}
\newtheorem*{theorem*}{Theorem}
\newtheorem*{conjecture*}{Conjecture}
\newtheorem*{theoremA}{Theorem A}
\newtheorem*{theoremB}{Theorem B}
\newtheorem{thm}{Theorem}[section]
\newtheorem{theorem}[thm]{Theorem}

\newtheorem{corollary}[thm]{Corollary}
\newtheorem{lemma}[thm]{Lemma}
\newtheorem*{lemma*}{Lemma}
\newtheorem{conjecture}[thm]{Conjecture}
\newtheorem{proposition}[thm]{Proposition}
\theoremstyle{definition}
\newtheorem{definition}[thm]{Definition}
\theoremstyle{remark}
\newtheorem*{remark}{Remark}

\input cyracc.def

\newdir{ (}{{}*!/-5pt/@^{(}}

\SelectTips{eu}{10}
\UseTips
\entrymodifiers={+!!<0pt,\fontdimen22\textfont2>}

\DeclareMathOperator{\Gal}{Gal}
\DeclareMathOperator{\rank}{rank}
\DeclareMathOperator{\corank}{corank}

\DeclareMathOperator{\Selm}{Sel}

\DeclareMathOperator{\img}{img}
\DeclareMathOperator{\coker}{coker}
\DeclareMathOperator{\Hom}{Hom}

\DeclareMathOperator{\res}{res}
\DeclareMathOperator{\red}{red}
\DeclareMathOperator{\cores}{cores}
\DeclareMathOperator{\infl}{inf}
\DeclareMathOperator{\dual}{dual}
\DeclareMathOperator{\Frob}{Frob}

\newcommand{\Q}{{\mathbb{Q}}}
\newcommand{\Z}{{\mathbb{Z}}}
\newcommand{\N}{{\mathbb{N}}}

\newcommand{\C}{{\mathbb{C}}}

\newcommand{\Fp}{{\mathbb{F}_p}}
\newcommand{\Zp}{{\mathbb{Z}_p}}

\newcommand{\Qp}{{\mathbb{Q}_p}}
\newcommand{\ilim}{\mathop{\varprojlim}\limits}
\newcommand{\dlim}{\mathop{\varinjlim}\limits}
\newcommand{\Sel}{{\Selm_p}}
\newcommand{\Selinf}{{\Selm_{p^{\infty}}}}

\newcommand{\cN}{\mathcal{N}}
\newcommand{\cG}{\mathcal{G}}
\newcommand{\cO}{\mathcal{O}}

\newcommand{\cy}[1]{\mathbb{Z}/#1\mathbb{Z}}
\newcommand{\fp}{\mathfrak{p}}

\newcommand{\overbar}[1]{\mkern 1.5mu\overline{\mkern-1.5mu#1\mkern-1.5mu}\mkern 1.5mu}
\newcommand{\dotcup}{\ensuremath{\mathaccent\cdot\cup}}
\newcommand{\joinrelshort}{\mathrel{\mkern-9mu}}
\newcommand{\shortlongrightarrow}{\relbar\joinrelshort\rightarrow}
\newcommand{\isomarrow}{\mathrel{\mathop{\setbox0\hbox{$\mathsurround0pt
        \shortlongrightarrow$}\ht0=0.7\ht0\box0}\limits
        ^{\sim\mkern2mu}}}

\begin{document}

\title{Selmer Groups and Anticyclotomic $\Zp$-extensions}

\author{Ahmed Matar}

\address{Department of Mathematics\\
         University of Bahrain\\
         P.O. Box 32038\\
         Sukhair, Bahrain}
\email{amatar@uob.edu.bh}

\begin{abstract}
Let $E/\Q$ be an elliptic curve, $p$ a prime and $K_{\infty}/K$ the anticyclotomic $\Zp$-extension of a quadratic imaginary field $K$ satisfying the Heegner hypothesis. In this paper we give a new proof to a theorem of Bertolini which determines the value of the $\Lambda$-corank of $\Selinf(E/K_{\infty})$ in the case where $E$ has ordinary reduction at $p$. In the case where $E$ has supersingular reduction at $p$ we make a conjecture about the structure of the module of Heegner points mod $p$. Assuming this conjecture we give a new proof to a theorem of Ciperiani which determines the value of the $\Lambda$-corank of $\Selinf(E/K_{\infty})$ in the case where $E$ has supersingular reduction at $p$.
\end{abstract}

\maketitle

\section{Introduction}

Let $E$ be an elliptic curve of conductor $N$ defined over $\Q$ and let $K$ be an imaginary quadratic field  with discriminant $d_K\neq -3,-4$ such that all the primes dividing $N$ split in $K/\Q$. We will denote the class number of $K$ by $h_K$. Now suppose $p \geq 5$ is a prime not dividing $Nd_Kh_K \varphi(Nd_K)$ (together with some additional restrictions listed in section 2.1).

Let $K_{\infty}/K$ be the anticylotomic $\Zp$-extension of $K$, $\Gamma=\Gal(K_{\infty}/K)$ and $K_n$ the unique subfield of $K_{\infty}$ containing $K$ such that $\Gal(K_n/K) \cong \cy{p^n}$

For any $n$ we let $\Selinf(E/K_n)$ denote the $p^{\infty}$-Selmer group of $E$ over $K_n$ defined by
$$\displaystyle 0 \longrightarrow \Selinf(E/K_n) \longrightarrow H^1(K_n, E[p^{\infty}])\longrightarrow \prod_v H^1(K_{n,v}, E)[p^{\infty}]$$
We also define the $p^{\infty}$-Selmer group of $E$ over $K_{\infty}$ as $\Selinf(E/K_{\infty})=\dlim\Selinf(E/K_n)$

Now let $T_p\Selinf(E/K_n)$ be the $p$-adic Tate module of $\Selinf(E/K_n)$. We will also be interested in the pro-$p$ Selmer group of $E$ over $K_{\infty}$ defined as $X_{p^{\infty}}(E/K_{\infty})=\ilim T_p\Selinf(E/K_n)$ where the inverse limit is taken over $n$ with respect to the corestriction maps

Finally, let $\Lambda=\Zp[[\Gamma]]$ be the Iwasawa algebra attached to $K_{\infty}/K$. Fixing a topological generator $\gamma \in \Gamma$ allows us to identify $\Lambda$ with the power series ring $\Zp[[T]]$. Throughout most of the paper we work ``mod $p$'' and so we will also consider the ``mod $p$'' Iwasawa algebra $\overbar{\Lambda}=\Lambda/p\Lambda=\Fp[[T]]$

Let $\pi: X_0(N) \to E$ be a modular parametrization of $E$ which maps the cusp $\infty$ of $X_0(N)$ to the origin of $E$ and let $E'$ be a strong Weil curve in the isogeny class of $E$ i.e. there exists a modular parametrization $\pi': X_0(N) \to E'$ which maps the cusp $\infty$ of $X_0(N)$ to the origin of $E'$ such that the induced map $\pi'_*: J_0(N) \to E'$ has a geometrically connected kernel.

Choosing an ideal $\cN$ of $\cO_K$ such that $\cO_K/\cN \cong \cy{N}$ allows us to define a family of Heegner points $\alpha_n \in E(K_n)$ using the modular parametrization $\pi$ and a family of Heegner points $\alpha'_n \in E'(K_n)$ using the modular parametrization $\pi'$ (see section 2). We will make the following conjecture

\begin{conjecture}\label{main_conjecture}
Assume that $p$ splits in $K/\Q$ and $E$ has good supersingular reduction at $p$ then the $\Gamma$-submodule of $E'(K_{\infty})/p$ generated by the Heegner points $\alpha'_n$ has $\overbar{\Lambda}$-corank greater than or equal to two.
\end{conjecture}

We will give strong evidence in support of this conjecture. See theorem \ref{Heegner_modules_rank_theorem} and the remarks following it.

In section 3 we will give a new proof to the following theorem of Bertolini \cite{Bertolini}

\begin{theoremA}
Assume that $E$ has good ordinary reduction at $p$, then $\Selinf(E/K_{\infty})$ has $\Lambda$-corank equal to 1 and $X_{p^\infty}(E/K_{\infty})$ is a free $\Lambda$-module of rank 1.
\end{theoremA}

In section 4 we will prove the following theorem

\begin{theoremB}
Assume that $p$ splits in $K/\Q$, $E$ has good supersingular reduction at $p$ and conjecture \ref{main_conjecture} is true, then $\Selinf(E/K_{\infty})$ has $\Lambda$-corank equal to 2 and $X_{p^{\infty}}(E/K_{\infty})=\{0\}$
\end{theoremB}

The fact that $\Selinf(E/K_{\infty})$ has $\Lambda$-corank equal to 2 when both $E$ has good supersingular reduction at $p$ and $p$ splits in $K/\Q$ was proven by Ciperiani \cite{Ciperiani}. However, assuming the above conjecture, our proof of this fact will be different.

Regarding the two theorems above, we should note that neither Bertolini nor Ciperiani require that $p$ does not divide $\varphi(Nd_K)$ whereas we assume it for our proof.

Our proofs to Theorems A and theorem \ref{Heegner_modules_rank_theorem} in section 4 (the latter theorem gives strong evidence supporting conjecture \ref{main_conjecture}) very much rely on the work of Cornut (\cite{Cornut} Theorem B) which proves that if $p \nmid \varphi(Nd_K)$ then the $\Fp$-vector span of $\{\sigma(\alpha'_n) \otimes 1 \, | \, \sigma \in \Gal(K_{\infty}/K) \, \text{and} \, n \geq 0\} \subset E'(K_{\infty})\otimes \Fp$ has infinite dimension.

Our method of proof is an adaptation of the technique of Bertolini and Darmon \cite{BD} to our Iwasawa theoretic setting. As mentioned above, throughout most of the paper will work with the ``mod $p$'' Iwasawa algebra $\overbar{\Lambda}=\Fp[[T]]$

Specifically, we let $X_p(E/K_{\infty})=X_{p^{\infty}}(E/K_{\infty})/p$. We will show in the ordinary case (theorem A) that the $\overbar{\Lambda}$-rank of $X_p(E/K_{\infty})$ is less than or equal to one and in the supersingular case (theorem B) that $X_p(E/K_{\infty})=\{0\}$.  It then follows in the ordinary case that the $\Lambda$-rank of $X_{p^{\infty}}(E/K_{\infty})$ is also less than or equal to one and in the supersingular case that $X_{p^{\infty}}(E/K_{\infty})=\{0\}$

Let us now define $Y_{p^{\infty}}(E/K_{\infty})=\ilim T_p\Selinf(E/K_{\infty})^{{\Gamma}^{p^n}}$ where the inverse limit is taken over $n$ with respect to the norm maps.

The control theorem in section 2.3 in the ordinary case (which is an easy consequence of Mazur's control theorem) gives that the restriction maps induce an isomorphism $X_{p^{\infty}}(E/K_{\infty}) \isomarrow Y_{p^{\infty}}(E/K_{\infty})$. But then as we will explain in section 2.3 $Y_{p^{\infty}}(E/K_{\infty})$ is a free $\Lambda$-module whose rank is equal to the $\Lambda$-corank $\Selinf(E/K_{\infty})$. This together with the simple observation that the $\Lambda$-corank of $\Selinf(E/K_{\infty})$ is greater than or equal to one will prove theorem A.

In the supersingular case the control theorem in section 2.3 gives that the restriction maps induce an injection $X_{p^{\infty}}(E/K_{\infty}) \hookrightarrow Y_{p^{\infty}}(E/K_{\infty})$ with cokernel of $\Lambda$-rank less than or equal to two. Since $X_{p^{\infty}}(E/K_{\infty})=\{0\}$ we get that the $\Lambda$-rank of $Y_{p^{\infty}}(E/K_{\infty})$ is less than or equal to two. It then follows that the $\Lambda$-corank of $\Selinf(E/K_{\infty})$ is also less than or equal to two since it is equal to the $\Lambda$-rank of $Y_{p^{\infty}}(E/K_{\infty})$. But by a well-known result in the supersingular case we know that the $\Lambda$-rank of $\Selinf(E/K_{\infty})$ is greater than or equal to two. This then proves theorem B.

\section{Preliminaries}

\subsection{Notation and Assumptions}

First we list the assumptions we need for theorems A and B. As in the introduction, $E$ is an elliptic curve of conductor $N$ defined over $\Q$ and $K$ be an imaginary quadratic field  with discriminant $d_K\neq -3,-4$ such that all the primes dividing $N$ split in $K/\Q$. We will denote the class number of $K$ by $h_K$. Throughout the paper we assume that $p \geq 5$ is a prime such that $p \nmid Nd_Kh_K \varphi(Nd_K)$. We will also assume that $\Gal(\Q(E[p])/\Q)=GL_2(\Fp)$. Asumming $E$ has no complex multiplication this excludes a finite number of primes by a theorem of Serre \cite{Serre}. In addition to these assumptions we further asuume the following for theorems A and B.

\noindent For theorem A we assume:\\

\noindent(1) $E$ has good ordinary reduction at $p$\\
(2) $p \nmid E(\Fp)$\\
(3) $a_p \not \equiv -1 \; (\text{mod } p)$ if $p$ is inert in $K/\Q$\\
(4) $a_p \not \equiv 2 \; (\text{mod } p)$ if $p$ splits in $K/\Q$\\

\noindent For theorem B we assume:\\

\noindent(1) $p$ splits in $K/\Q$\\
(2) $E$ has good supersingular reduction at $p$\\

Regarding the assumptions for theorem A, lets assume that $p$ splits in $K/\Q$. Then if $p > 7$, conditions (1), (2) and (4) are equivalent to $a_p \neq 0,1,2$ by the Hasse bound on $a_p$. As explained in \cite{Bertolini} pg. 166 the set of primes $p$ such that $a_p \neq 0,1,2$ has density 1 . We get a similar conclusion when $p$ is inert in $K/\Q$.

We will now explain the notation that we will use throught the paper. We fix a complex conjugation $\tau$ on $\overbar{\Q}$ (the algebraic closure of $\Q$). Given a $\Z[\frac{1}{2}][\tau]$-module $M$, we have a decomposition $M=M^+ \oplus M^-$ where $M^+$ and $M^-$ denotes the submodule on which $\tau$ acts as $+1$, respectively $-1$. Also, if $x \in M$ and $X \subset M$, we let

$$x^{\pm}=\frac{1}{2}(x\pm \tau x)$$
$$X^{\pm}=\{x^{\pm} \; | \; x \in X\}$$\\
For any $m$ we let $K[m]$ denote the ring class field of $K$ of conductor $m$. Let $K[p^{\infty}]=\cup_{n \geq 1} K[p^n]$. Then  $\Gal(K[p^{\infty}]/K)$ is isomorphic to $\Zp \times \Delta$, where $\Delta$ is a finite abelian group. The unique $\Zp$-extension that is contained in $K[p^{\infty}]/K$ is the anticyclotomic $\Zp$-extension of $K$ which we will denote by $K_{\infty}/K$. We let $K_n$ be the subextension of $K_{\infty}$ of degree $p^n$ over $K$

Let $\Gamma=\Gal(K_{\infty}/K)$. We will write $\Gamma_n$ for the Galois group $\Gal(K_{\infty}/K_n)=\Gamma^{p^n}$, $G_n$ for the Galois group $\Gal(K_n/K)=\Gamma/\Gamma_n$ and $R_n$ for the group ring $\Fp[G_n]$

We let $\Lambda=\Zp[[\Gamma]]$ be the Iwasawa algebra attached to $K_{\infty}/K$. Fixing a topological generator $\gamma \in \Gamma$ allows us to identify $\Lambda$ with the power series ring $\Zp[[T]]$

We will also work with the ``mod $p$'' Iwasawa algebra $\overbar{\Lambda}=\Lambda/p\Lambda=\Fp[[T]]$. Note that $\overbar{\Lambda}$ is a PID.

Let us now define the Selmer groups we will be working with: If $L/\Q$ is any algebraic extension we let $\Selinf(E/L)$ denote the $p^{\infty}$-Selmer group of $E$ over $L$ defined by
$$\displaystyle 0 \longrightarrow \Selinf(E/L) \longrightarrow H^1(L, E[p^{\infty}])\longrightarrow \prod_v H^1(L_v, E)[p^{\infty}]$$

We will also be working with the $p$-Selmer group $\Sel(E/L)$ defined by

$$\displaystyle 0 \longrightarrow \Sel(E/L) \longrightarrow H^1(L, E[p])\longrightarrow \prod_v H^1(L_v, E)[p]$$

Finally, if $\ell$ is a rational prime and $F$ is a number field we define

\begin{flalign*}
\qquad \qquad &E(F_{\ell})/p:=\oplus_{\lambda|\ell}E(F_{\lambda})/p&\\[0.5em]
&H^1(F_{\ell}, E[p]):=\oplus_{\lambda|\ell}H^1(F_{\lambda}, E[p])&\\[0.5em]
&H^1(F_{\ell}, E)[p]:=\oplus_{\lambda|\ell}H^1(F_{\lambda}, E)[p]
\end{flalign*}\\
\noindent where the sum is taken over all primes of $F$ dividing $\ell$

With this notation we let $\res_{\ell}$ be the localization map:

\begin{flalign*}
\qquad \qquad &\res_{\ell}: E(F)/p \to E(F_{\ell})/p&\\[0.5em]
&\res_{\ell}: H^1(F, E[p]) \to H^1(F_{\ell}, E[p])&\\[0.5em]
&\res_{\ell}: H^1(F, E)[p] \to H^1(F_{\ell}, E)[p]
\end{flalign*}\\
If $F=K_n$, with the above notation we let $K_{n,\ell}$ denote $F_{\ell}$

We will frequently write $\dlim$ (resp. $\ilim$) for $\dlim_n$ (resp. $\ilim_n$) as out limits are taken over $n$.

\subsection{Heegner points and Kolyvagin classes}
We fix a modular parametrization $\pi: X_0(N) \to E$ which maps the cusp $\infty$ of $X_0(N)$ to the origin of $E$ (see \cite{Wiles} and \cite{BDCT})
We have assumed that every prime dividing $N$ splits in $K/\Q$. It follows that we can choose and ideal $\cN$ such that $\cO_K/\cN \cong \Z/N\Z$. Let $m$ be an integer that is relatively prime to $Nd_K$ and let $\cO_m = \Z + m\cO_K$ be the order of conductor $m$ in $K$. The ideal $\cN_m=\cN \cap \cO_m$ satisfies $\cO_m/\cN_m \cong \Z/N\Z$ and therefore the natural projection of complex tori:

$$\C/\cO_m \to \C/\cN_m^{-1}$$\\
is a cyclic $N$-isogeny, which corresponds to a point of $X_0(N)$. Let $\alpha[m]$ be its image under the modular parametrization $\pi$. From the theory of complex multiplication we have that $\alpha[m] \in E(K[m])$ where $K[m]$ is the ring class field of $K$ of conductor $m$.

Then as we have assumed throughout the paper that the class number of $K$ is not divisible by $p$, it follows for any $n$ that $K[p^{n+1}]$ is the ring class field of minimal conductor that that contains $K_n$. We now define $\alpha_n \in E(K_n)$ to be the trace from $K[p^{n+1}]$ to $K_n$ of $\alpha[n]$

Let $R_n\alpha_n$ denote the $R_n$-submodule of $H^1(K_n, E[p])$ generated by the image of $\alpha_n$ under the map

$$E(K_n) \to H^1(K_n, E[p])$$\\
We have that $E(K_{\infty}[p^{\infty}])=\{0\}$ by corollary \ref{Ep_corollary} of the next section. This implies that the restriction map for $m\geq n$

$$H^1(K_n, E[p]) \to H^1(K_m, E[p])$$\\
\noindent is injective and therefore allows us to view $R_n\alpha_n$ as a submodule of $H^1(K_m, E[p])$

From section 3.3 of \cite{PR} we have

\begin{align}
&\left\{
\begin{aligned}
\text{Tr}_{K_1/K}(\alpha_1) &= (a_p-a_p^{-1}(p+1))\alpha_0 \quad \text{if $p$ is inert in $K/\Q$}\\
\text{Tr}_{K_1/K}(\alpha_1) &= (a_p-(a_p-2)^{-1}(p-1))\alpha_0 \quad \text{if $p$ splits in $K/\Q$}\\
\end{aligned}
\right.\\[10pt]
&\text{Tr}_{K_{n+1}/K_n}(\alpha_{n+1})=a_p\alpha_n-\alpha_{n-1} \quad \text{for $n \geq 1$}
\end{align}\\
In the ordinary case (theorem A), our assumptions together with (1) and (2) allow us to conclude that (see \cite{Bertolini} prop 2.1.4) $\text{Tr}_{K_{n+1}/K_n}(\alpha_{n+1})=u\alpha_n$ for some unit $u \in R_n$. This implies that $R_n\alpha_n \subset R_{n+1}\alpha_{n+1}$

In the supersingular case (theorem B), the fact that $E$ has supersingular reduction at $p$ and $p \geq 5$ implies that $a_p=0$ so (2) becomes $\text{Tr}_{K_{n+1}/K_n}(\alpha_{n+1})=-\alpha_{n-1}$. This then implies that $R_n\alpha_n \subset R_{n+2}\alpha_{n+2}$

As in \cite{CW} section 2.5.1, we now describe the construction of Kolyvagin classes over ring class fields following \cite{BD}. However, we should note that our definition of Kolyvagin classes differs slightly from the one in \cite{CW} and \cite{Bertolini}. First let us make the following definition

\begin{definition} A rational prime $\ell$ is called a \textit{Kolyvagin prime} if\\
(i) $\ell$ is relatively prime to $pNd_K$\\
(ii) $\Frob_{\ell}(K(E[p])/\Q)=[\tau]$
\end{definition}

Let $r$ be a squarefree product of Kolyvagin primes. For any $n$ let $K_n[r]$ denote the field $K_nK[r]$. We now define $\alpha_n(r)$ to be the trace of $\alpha[rp^{n+1}]$ from $K[rp^{n+1}]$ to $K_n[r]$

Let $G_{n,r}=\Gal(K_n[r]/K_n[1])$ and let $G_{n,\ell}=\Gal(K_n[\ell]/K_n[1])$. By class field theory $G_{n,r}=\prod_{\ell|r}G_{n,\ell}$ and $G_{n,\ell}$ is cyclic of order $\ell+1$. Let $\sigma_{\ell}$ be a generator of $G_{n,\ell}$ and define $D_{\ell} :=\sum_{i=1}^{\ell} i\sigma_{\ell}^i \in \cy{p}[G_{n,\ell}]$ and $D_r:=\prod_{\ell|r}D_{\ell} \in \cy{p}[G_{n,r}]$. Then one can show that $D_r\alpha_n(r)$ belongs to $(E(K_n[r]/p))^{G_{n,r}}$ (see \cite{BD} lemma 3.3). It follows that $\text{Tr}_{K_n[1]/K_n} D_r \alpha_n(r) \in (E(K_n[r])/p)^{\cG_n,r}$ where $\cG_{n,r}=\Gal(K_n[r]/K_n)$. Now consider the commutative diagram

\begin{equation}\label{kolyvagin_classes_diagram}
\xymatrix{
&&& 0 \ar[d]\\
&&& H^1(K_n[r]/K_n, E)[p] \ar[d]_{\infl}\\
0 \ar[r] & E(K_n)/p \ar[d] \ar[r]^-{\phi} & H^1(K_n, E[p]) \ar[d]^{\rotatebox{90}{$\text{\~{}}$}}_{\res} \ar[r] & H^1(K_n, E)[p] \ar[d]_{\res} \ar[r] & 0\\
0 \ar[r] & (E(K_n[r])/p)^{\cG_{n,r}} \ar[r]^-{\phi_r} & H^1(K_n[r], E[p])^{\cG_{n,r}} \ar[r] & H^1(K_n[r],E)[p]^{\cG_{n,r}}
}
\end{equation}

Let $c_n(r) \in H^1(K_n, E[p])$ be so that

$$\phi_r(\text{Tr}_{K_n[1]/K_n} D_r\alpha_n(r))=\text{Res}(c_n(r))$$\\
and let $d_n(r)$ be the image of $c_n(r)$ in $H^1(K_n, E)[p]$. Note that $c_n(1) = \phi(\alpha_n)$\\
These Kolyvagin classes have the following properties:
\begin{enumerate}
\item Let $-\epsilon$ denote the sign of the functional equation of the L-function of $E/\Q$, $f_r$ be the number of prime divisors of $r$ and $\tau$. We have $\tau\alpha_n=\epsilon g^{i_{n,1}}\alpha_n + \beta_n$ with $\beta_n \in E(K_n)_{\text{tors}}$, $g$ a generator of $\Gal(K_{\infty}/K)$ and $i_{n,1} \in \{0,...,p^n-1\}$. Moreover, $\tau$ acts on $H^1(K_n, E[p])$ and we can deduce that $\tau c_n(r)=\epsilon_r g^{i_{n,r}}c_n(r)$ where $\epsilon_r=(-1)^{f_r}\epsilon$ and $i_{n,r}\in \{0,...,p^n-1\}$
\item If $v$ is a rational prime that does not divide $r$, then $d_n(r)_{v_n}=0$ in $H^1(K_{v_n}, E)[p]$ for all primes of $K_n \, \, v_n|v$
\item If $\ell|r$, there exists a $G_n$-equivariant and a $\tau$-antiequivariant isomorphism:
$$\psi_{n, \ell}: H^1(K_{n,\ell}, E)[p] \to E(K_{n,\ell})/p$$
such that $\psi_{n,\ell}(\res_{\ell}d_n(r))=\res_{\ell}(c_n(r/\ell))$\\
If we let $\res_n$ denote the restriction maps $H^1(K_{n,\ell}, E)[p] \to H^1(K_{n+1, \ell}, E)[p]$ and $E(K_{n,\ell})/p \to E(K_{n+1, \ell})/p$, then we have \\$\psi_{n+1, \ell} \circ \res_n = \res_n \circ \psi_{n, \ell}$
\item In the ordinary case, just as $R_n\alpha_n \subset R_{n+1}\alpha_{n+1}$ we also have $R_nc_n(r) \subset R_{n+1}c_{n+1}(r)$ and $R_nd_n(r) \subset R_{n+1}d_{n+1}(r)$\\
    In the supersingular case, just as $R_n\alpha_n \subset R_{n+2}\alpha_{n+2}$ we also have $R_nc_n(r) \subset R_{n+2}c_{n+2}(r)$ and $R_nd_n(r) \subset R_{n+2}d_{n+2}(r)$
\end{enumerate}

We end this section with the following proposition. Just as in \cite{Gross} prop. 3.7 we have

\begin{proposition}\label{Eichler-Shimura_proposition}
Every prime $\lambda_{\ell m}$ of $K_n[\ell m]$ above $\ell$ lies above a unique prime $\lambda_m$ of $K_n[m]$ and we have $\alpha_n(\ell m) \equiv \Frob(\lambda_m/\ell)\alpha_n(m) \mod \lambda_{\ell m}$
\end{proposition}

\subsection{Preliminary Results}

In this section, we collect some preliminary results that will be used in the proofs of theorems A and B. First we have the following important lemma

\begin{lemma}\label{GL2_lemma}
The extensions $\Q(E[p])/\Q$ and $K_{\infty}/\Q$ are linearly disjoint. In particular, $\Gal(K_n(E[p])/K_n)$ is isomorphic to $GL_2(\Fp)$ for any $n$
\end{lemma}
\begin{proof}
First we show that $\Q(E[p])/\Q$ and $K/\Q$ are disjoint: the extension $\Q(E[p])/\Q$ is ramified only at primes dividing $Np$. This implies that the intersection of $\Q(E[p])$ and $K$ is an unramified extension of $\Q$ and therefore must be $\Q$ itself. Therefore $\Q(E[p])/\Q$ and $K/\Q$ are disjoint and we have $\Gal(K(E[p])/K)=\Gal(\Q(E[p])/\Q)=GL_2(\Fp)$

Now we show that $K(E[p])/K$ and $K_{\infty}/K$ are disjoint. If they were not disjoint then $\Gal(K(E[p])/K)=GL_2(\Fp)$ would have a normal subgroup $N$ of index $p$. As $SL_2(\Fp)$ has index $p-1$ in $GL_2(\Fp)$, therefore by order considerations we must have that $N \cap SL_2(\Fp)$ is a subgroup of $SL_2(\Fp)$ of both order and index greater than 2. But this is impossible as $PSL_2(\Fp)$ is simple for $p \geq 5$
\end{proof}

\begin{corollary}\label{Ep_corollary}
We have $E(K_{\infty})[p^{\infty}]=\{0\}$
\end{corollary}

Now for any $n$ and any rational prime $\ell$, local Tate duality gives non-degenerate pairing (see \cite{Gross} prop. 7.5)

\begin{equation}
\langle \; , \; \rangle_{\ell}: E(K_{n,\ell})/p \times H^1(K_{n,\ell}, E)[p] \to \Fp
\end{equation}\\
This identifies $H^1(K_{n, \ell}, E)[p]$ with $(E(K_{n,\ell})/p)^{\dual}$

Moreover, if $a \in E(K_{n+1,\ell})/p$ and $b \in H^1(K_{n,\ell}, E)[p]$, then a property of Tate local duality gives $\langle\cores a, b\rangle=\langle a, \res b\rangle$ where

$$\res: H^1(K_{n, \ell}, E)[p] \to H^1(K_{n+1, \ell}, E)[p]$$\\
is the restriction map and

$$\cores : E(K_{n+1, \ell})/p \to E(K_{n, \ell})/p$$\\
is the corestriction map (the norm map). Therefore Tate local duality induces an isomorphism

\begin{equation}\label{Tatelocalduality_isomorphism}
\dlim H^1(K_{n,\ell}, E)[p] \cong (\ilim E(K_{n,\ell})/p)^{\dual}
\end{equation}\\
where the direct limit is taken over $n$ with respect to the restriction maps and the inverse limit is taken over $n$ with respect to the corestriction maps

The $p$-Selmer group $\Sel(E/K_n)$ consists of the cohomology classes $s \in H^1(K_n, E[p])$ whose restrictions $\res_v(s)\in H^1(K_{n,v}, E[p])$ belong to $E(K_{n,v})/p$ for all primes $v$ of $K_n$ where we view $E(K_{n,v})/p$ as a subspace of $H^1(K_{n,v}, E[p])$ using the Kummer sequence

$$\displaystyle 0 \longrightarrow E(K_{n,v})/p \longrightarrow H^1(K_{n,v}, E[p]) \longrightarrow H^1(K_{n,v}, E)[p] \longrightarrow 0$$\\
Therefore we have a map

\begin{equation}
\res_{\ell}: \Sel(E/K_n) \to E(K_{n,\ell})/p
\end{equation}\\
Let $T_p \Selinf(E/K_n)$ denote the $p$-adic Tate module of $\Selinf(E/K_n)$. We now define $X_{p^{\infty}}(E/K_{\infty}):=\ilim T_p\Selinf(E/K_n)$ where the inverse limit is taken over $n$ with respect to the corestriction maps. We will denote $X_{p^{\infty}}(E/K_{\infty})/p$ by $X_p(E/K_{\infty})$. By corollary \ref{Ep_corollary} for any $n$ we have an isomorphism $\Sel(E/K_n) \isomarrow \Selinf(E/K_n)[p]$. This allows us to view $X_p(E/K_{\infty})$ as being a subgroup of $\ilim \Sel(E/K_n)$. We will use this fact throughout the paper.

The map $\res_{\ell}$ then induces a map

\begin{equation}
\res_{\ell}: X_p(E/K_{\infty}) \to \ilim E(K_{n,\ell})/p
\end{equation}\\
Dualizing this map and using the isomorphism (\ref{Tatelocalduality_isomorphism}) above we get a map

$$\uppsi_{\ell}: \dlim H^1(K_{n,\ell}, E)[p] \to X_p(E/K_{\infty})^{\dual}$$\\
Recall that $\overbar{\Lambda}$ denotes the ``mod $p$'' Iwasawa algebra $\Lambda/p\Lambda=\Fp[[T]]$. Our goal in theorems A and B is to determine the $\overbar{\Lambda}$-rank of $X_p(E/K_{\infty})$ This will be done by determining the $\overbar{\Lambda}$-corank of the image of $\Psi_{\ell}$ for various primes $\ell$. To do this we will need the following important observation

\begin{proposition}\label{Iwasawa_rank_proposition}
If $\ell$ is a Kolyvagin prime, then $\dlim H^1(K_{n,\ell}, E)[p]$ is a cofree $\overbar{\Lambda}$-module of rank 2
\end{proposition}
\begin{proof}
As $\ell$ is inert in $K/\Q$ and $\ell \neq p$, it follows that $\ell$ splits completely in the anticyclotomic $\Zp$-extension $K_{\infty}/K$ and so the $\Gamma_n$-invariants of $\dlim H^1(K_{n,\ell}, E)[p]$ is equal to $H^1(K_{n,\ell}, E)[p]$ which is isomorphic to the dual of $E(K_{n,\ell})/p =\oplus_{\lambda_n|\ell}E(K_{n,\lambda_n})/p$ by local Tate duality. For any $\lambda_n |\ell$ we have by Mattuck's theorem that $E(K_{n,\lambda_n})\cong\Z_{\ell}^2 \times T$ where $T$ is a finite group. This together with the fact that $\ell$ splits in $K(E[p])/K$ implies that $E(K_{n, \lambda_n})/p=\cy{p}\times\cy{p}$ and so the $\Gamma_n$-invariants of $\dlim H^1(K_{n, \ell}, E)[p]$ has $\Fp$-dimension $2p^n$ which implies that the $\overbar{\Lambda}$-corank of $\dlim H^1(K_{n, \ell}, E)[p]$ is equal to 2. It also follows that $\dlim H^1(K_{n, \ell}, E)[p]$ is cofree as a $\overbar{\Lambda}$-module, for if it was not cofree then its $\Gamma$-invariants would have $\Fp$-dimension $2p +c$ for some positive integer $c$.
\end{proof}

Now for any $n$ let $L_n=K_n(E[p])$ and $\cG_n=\Gal(L_n/K_n)$ which is isomorphic to $GL_2(\Fp)$ by lemma \ref{GL2_lemma}, then we have the following proposition (\cite{Gross} prop. 9.1)

\begin{proposition}\label{res_isom_prop}
The restriction map induces an isomorphism:
$$\res: H^1(K_n, E[p]) \isomarrow H^1(L_n, E[p])^{\cG_n} =\Hom_{\cG_n}(\Gal(\overbar{\Q}/L_n), E[p])$$
\end{proposition}

From the above proposition we get a pairing

\begin{equation}
[\; , \;]: H^1(K_n, E[p]) \times \Gal(\overbar{\Q}/L_n) \to E[p]
\end{equation}\\
Now assume that $S_n \subset H^1(K_n, E[p])$ is a finite subgroup. Let $\Gal_{S_n}(\overbar{\Q}/L_n)$ be the subgroup consisting of $\rho \in \Gal(\overbar{\Q}/L_n)$ such that $[s,\rho]=0$ for all $s \in S_n$ and let $L_{S_n}$ be the fixed field of $\Gal_{S_n}(\overbar{\Q}/L_n)$. Then $L_{S_n}/K_n$ is a finite Galois extension and the above pairing induces a nondegenerate pairing

\begin{equation}\label{pairing}
[\; , \;]: S_n \times \Gal(L_{S_n}/L_n) \to E[p]
\end{equation}\\
We have the following lemma

\begin{lemma}
The extensions $L_{S_n}/K_n$ and $K_{\infty}/K_n$ are disjoint
\end{lemma}
\begin{proof}

By lemma \ref{GL2_lemma} the extensions $L_n/K_n$ and $K_{\infty}/K_n$ are disjoint. Therefore we have that $\Gal(L_nK_{\infty}/L_n)\cong\Gal(K_{\infty}/K_n)$. We now show that $L_nK_{\infty}/L_n$ and $L_{S_n}/L_n$ are disjoint. If they were not disjoint then $\Gal(L_{S_n}/L_n)$ would a have a quotient of order $p$ on which $\cG_n$ acts trivially. But $\Gal(L_{S_n}/L_n) \cong E[p]^r$ where $r=\dim_{\Fp}S_n$ is a semisimple $\cG_n$-module (see \cite{Gross} prop. 9.3). Therefore any quotient of $\Gal(L_{S_n}/L_n)$ is isomorphic to $E[p]^s$ for some $s \le r$ (\cite{Bourbaki} cor. 4.3) so $L_{S_n}/L_n$ and $L_nK_{\infty}/L_n$ are indeed disjoint which completes the proof.
\end{proof}

We now assume that for some $n_0$ we have a finite subgroup $S_{n_0} \subset H^1(K_{n_0}, E[p])$ that is stable under $\Gal(K_{n_0}/\Q)$. Then $L_{S_{n_0}}/\Q$ is a finite Galois extension. Let $V=\Gal(L_{S_{n_0}}/L_{n_0})$. Given a subset $U$ of $V$ we define

$$\mathscr{L}(U)=\{\ell \; \text{rational prime}\; | \;\ell \nmid pN\; \text{and} \; \Frob_{\ell}(L_{S_{n_0}}/\Q)=[\tau u] \; \text{for} \; u \in U \}$$

Note that every $\ell \in \mathscr{L}(U)$ is a Kolyvagin prime.

\begin{proposition}\label{generating_Xp_proposition}
If $U^+$ generates $V^+$, then $\img \uppsi_{\ell}$ with $\ell$ ranging over $\mathscr{L}(U)$ generate $X_p(E/K_{\infty})^{\; \dual}$
\end{proposition}
\begin{proof}
Let $s=(s_n) \in X_p(E/K_{\infty})$ with $s_n \in \Sel(E/K_n)$. To prove the proposition, it suffices to show that $\res_{\ell}(s)=0$ for all $\ell \in \mathscr{L}(U)$ implies  $s=0$ i.e. we must show for any $n$ that $\res_{\ell}(s_n)=0$ implies that $s_n=0$. Of course it suffices to show this for all $n \geq n_0$

Let $n \geq n_0$. By the previous lemma the extensions $L_{S_{n_0}}/K_{n_0}$ and $K_n/K_{n_0}$ are disjoint. Therefore, $V_n:=\Gal(L_{S_{n_0}}K_n/L_n)$ may be identified via restriction with $V=\Gal(L_{S_{n_0}}/L_{n_0})$. Let $U_n$ be the subset of $V_n$ corresponding to $U$. Then $U_n^+$ generates $V_n^+$. Moreover, if $\ell \in \mathscr{L}(U)$ with $\Frob_{\ell}(L_{S_{n_0}}/\Q)=[\tau u]$ with $u \in U$, then as $\ell$ is inert in $K/\Q$ and $\ell \neq p$ therefore $\ell$ splits completely in the anticyclotomic $\Zp$-extension $K_{\infty}/K$. This implies that $\Frob_{\ell}(L_{S_{n_0}}K_n/\Q)=[\tau u']$ where $u'$ is the element of $U_n$ corresponding to $u$.

We are now ready to prove the result. As above, we will show that $\res_{\ell}(s_n)=0$ implies that $s_n=0$. Without loss of generality we may assume that $s_n$ is in an eigenspace for the action of $\tau$

By proposition \ref{res_isom_prop} the restriction map induces an isomorphism
$$\res: H^1(K_n, E[p]) \isomarrow H^1(L_n, E[p])^{\cG_n} =\Hom_{\cG_n}(\Gal(\overbar{\Q}/L_n), E[p])$$

Using this isomorphism we identify $s_n$ with its image in $\Hom_{\cG_n}(\Gal(\overbar{\Q}/L_n), E[p])$

Now choose a Galois extension $M$ of $\Q$ containing $L_{S_{n_0}}K_n$ such that: (i) $\Gal(M/L_n)$ is abelian and (ii) $s_n$ factors through $\Gal(M/L_n)$. Let $x \in \Gal(M/L_n)$ be such that $x|_{L_{S_{n_0}}K_n} \in U_n$. By the Chebotarev density theorem, we may find $\ell \in \mathscr{L}(U)$ such that $\Frob_{\ell}(M/\Q)=[\tau x]$

That fact that $\res_{\ell}(s_n)=0$ means that $s_n(\Frob_{\lambda}(M/L_n))=0$ for all primes $\lambda$ of $L_n$ above $\ell$. Since $\Frob_{\ell}(M/\Q)=[\tau x]$ therefore for any prime $\lambda$ of $L_n$ above $\ell$ we have $\Frob_{\lambda}(M/L_n)=(\tau x)^2=x^{\tau}x=(x^+)^2$ and hence $s_n(x^+)=0$.

Since $U_n^+$ generates $V_n^+$ therefore the homomorphism vanishes on $\Gal(M/L_n)^+$ and hence as $s_n$ is in an eigenspace for the action of $\tau$ this implies that the image of $s_n$ is contained in a $\tau$-eigenspace of $E[p]$. In particular, it is a proper $\cG_n$-submodule of $E[p]$. Hence it is trivial since $\cG_n=GL_2(\Fp)$. Therefore $s_n=0$
\end{proof}

The following proposition will be an important tool to finding relations in $X_p(E/K_{\infty})^{\dual}$

\begin{proposition}\label{global_duality_proposition}

For any $n$, if $s \in \Sel(E/K_n)$ and $\gamma \in H^1(K_n, E)[p]$, then

$$\sum_{\ell} \langle \res_{\ell} s, \; \res_{\ell} \gamma \rangle_{\ell} =0$$

where the sum is taken over all the rational primes

\end{proposition}

The proposition is an immediate consequence of the global reciprocity law for elements in the Brauer group of $K_n$ (\cite{NSW} th. 8.1.17), taking into account the definition of local Tate duality (\emph{loc. cit.} th. 7.2.6)

We now define $Y_{p^{\infty}}(E/K_{\infty})=\ilim T_p\Selinf(E/K_{\infty})^{\Gamma_n}$ where the inverse limit is taken over $n$ with respect to the norm maps

The restriction maps $\res: \Selinf(E/K_n) \to \Selinf(E/K_{\infty})^{\Gamma_n}$ induce a map

$$\Xi: X_{p^{\infty}}(E/K_{\infty}) \to Y_{p^{\infty}}(E/K_{\infty})$$\\
In the final part of this section, we will prove an Iwasawa-theoretic control theorem which determines the $\Lambda$-rank of the kernel and cokernel of this restriction map. As explained in the introduction, this control theorem will allow us to deduce the value of $\Lambda$-corank of $\Selinf(E/K_{\infty})$ from the value of the $\Lambda$-rank of $X_{p^{\infty}}(E/K_{\infty})$

\begin{theorem}
Consider the map $\Xi$ induced by restriction

$$\Xi: X_{p^{\infty}}(E/K_{\infty}) \to Y_{p^{\infty}}(E/K_{\infty})$$\\

\begin{enumerate}[(a)]
\item If $E$ has good ordinary reduction at $p$, then $\Xi$ is an isomorphism\\
\item If $E$ has good supersingular reduction at $p$ and $p$ splits in $K/\Q$, then $\Xi$ is an injection and $\rank_{\Lambda}(\coker \Xi) \le 2$
\end{enumerate}

\end{theorem}
\begin{proof}
First we prove part (a): Assume that $E$ has good ordinary reduction at $p$. From Mazur's control theorem (\cite{Mazur}; see also \cite{Gb_LNM} and \cite{Gb_IIC}) using the fact that $E(K_{\infty})[p^{\infty}]=\{0\}$ (corollary \ref{Ep_corollary}) we get that for any $n$ the restriction map

$$\res_n: \Selinf(E/K_n) \to \Selinf(E/K_{\infty})^{\Gamma_n}$$\\
is an injection with finite cokernel. Part (a) then follows from this by taking Tate modules and then inverse limits over $n$.

Now we prove part (b): Assume that $E$ has good supersingular reduction at $p$ and $p$ splits in $K/\Q$. Define $S=\{p\} \cup \{l \; \text{prime}: l|N\}$. For any $n$, with this set $S$, we define $S_n$ to be the set of primes of $K_n$ above those in $S$ and $S_{\infty}$ to be the primes of $K_{\infty}$ above those in $S$. Now define $K_S$ to be the maximal extension of $K$ unramified outside $S$, $G_S(K_n)=\Gal(K_S/K_n)$ and $G_S(K_{\infty})=\Gal(K_S/K_{\infty})$. Note that since we have assumed all the primes dividing $N$ to split in $K/\Q$ therefore it follows from theorem 2 of \cite{Brink} that the set $S_{\infty}$ is finite.

For any $K_n$ and any $m$ it is well-konown that the $p^m$-Selmer group $\Selm_{p^m}(E/K_n)$ may be defined as

$$\displaystyle 0 \longrightarrow \Selm_{p^m}(E/K_n) \longrightarrow H^1(G_S(K_n), E[p^m])\longrightarrow \prod_{v\in S_n} H^1(K_{n,v} E)[p^m]$$

We may also define $\Selm_{p^m}(E/K_{\infty})$ as

$$\displaystyle 0 \longrightarrow \Selm_{p^m}(E/K_{\infty}) \longrightarrow H^1(G_S(K_{\infty}), E[p^m])\longrightarrow \prod_{v\in S_{\infty}} H^1(K_{\infty,v} E)[p^m]$$

For any $n$ and $m$ consider the following commutative diagram:

\begin{equation}
\xymatrix{
0 \ar[r] & \Selm_{p^m}(E/K_{\infty})^{\Gamma_n} \ar[r] & H^1(G_S(K_{\infty}), E[p^m])^{\Gamma_n} \ar[r]^-{\psi_{{\infty},m}}
& \underset{v \in S_{\infty}}{\bigoplus} H^1(K_{\infty,v}, E)[p^m]^{\Gamma_n} \\
0 \ar[r] & \Selm_{p^m}(E/K_n) \ar[u]_{s_{n,m}} \ar[r] & H^1(G_S(K_n), E[p^m]) \ar[u]_{h_{n,m}} \ar[r]^-{\psi_{n,m}}
& \underset{v \in S_n} {\bigoplus} H^1(K_{n,v}, E)[p^m] \ar[u]_{g_{n,m}}
}
\end{equation}

The vertical maps in the above diagram are restriction. Let us note a few things related to this diagram:\\

(1) The maps $h_{n,m}$ are isomorphisms: This follows from the fact that $H^1(\Gamma_n, E(K_{\infty})[p^m])$ and $H^2(\Gamma_n, E(K_{\infty})[p^m])$ are both trivial because $E(K_{\infty})[p^{\infty}]=\{0\}$ (corollary \ref{Ep_corollary})\\

(2) For any $v \in S_{\infty}$ above $p$ we have $H^1(K_{{\infty},v}, E)[p^{\infty}]=\{0\}$: This is equaivalent to $E(K_{\infty,v}) \otimes \Qp/\Zp \isomarrow H^1(K_{\infty,v}, E[p^{\infty}])$ with the map being the usual one from the Kummer sequence. The result follows from \cite{CG} cor. 3.2 as explained in \cite{Gb_LNM} pg. 70. Note that the fact that $E$ has supersingular reduction at $p$ is crucial for this result\\

(3) For any $v \in S_n$ not dividing $p$ we have that $H^1(K_{n,v}, E)[p^{\infty}]$ is finite: This follows from 2 facts. First, by Tate duality for abelian varieties over local fields (\cite{Milne} cor. 3.4) we have that $H^1(K_{n,v}, E)[p^{\infty}]$ is isomorphic to $\ilim E(K_{n,v})/p^n$. Secondly, if $l$ is the rational prime lying below $v$, then by Mattuck's theorems we have that $E(K_{n,v}) \cong \Z_l^d \times T$ where $d=[K_{n,v}:\Q_l]$ and $T$ is a finite group. Therefore it follows from these 2 facts that $H^1(K_{n,v}, E)[p^{\infty}]$ is the (finite) $p$-primary subgroup of $E(K_{n,v})$\\

(4) Let $\fp_1$ and $\fp_2$ be the primes of $K$ above $p$. Since we have assumed that the class number of $K$ is relatively prime to $p$, therefore both $\fp_1$ and $\fp_2$ are totally ramified in $K_{\infty}/K$. So in particular there are only 2 primes $\fp_{n,1}$ and $\fp_{n,2}$ of $K_n$ above $p$ and 2 primes $\fp_{\infty,1}$ and $\fp_{\infty,2}$ of $K_{\infty}$ above $p$\\

(5) For any $n$ and $m$ we have isomorphisms $H^1(K_n, E[p^m]) \isomarrow H^1(K_n, E[p^{\infty}])[p^m]$ and $\Selm_{p^m}(E/K_n) \isomarrow \Selinf(E/K_n)[p^m]$. A similar observation applies to these groups over $K_{\infty}$: This follows from corollary \ref{Ep_corollary}\\

Let $\tilde{S}_{\infty}=S_{\infty} \backslash \{\fp_{\infty,1}, \fp_{\infty,2}\}$ (see (4)). Taking the points (2)-(5) into consideration, we take the inverse limit of the objects in the diagram above over $m$ (using the multiplication by $p$ map) and then over $n$ (using the corestriction map for the bottom row and the norm map for the top row) to obtain the following diagram

\begin{equation}
\xymatrix{
0 \ar[r] & Y_{p^{\infty}}(E/K_{\infty}) \ar[r] & \ilim T_p H^1(G_S(K_{\infty}), E[p^{\infty}])^{\Gamma_n} \ar[r]^-{\upphi}
& \underset{v \in \tilde{S}_{\infty}}{\bigoplus} \ilim T_p H^1(K_{\infty,v}, E)[p^{\infty}]^{\Gamma_n} \\
0 \ar[r] & X_{p^{\infty}}(E/K_{\infty}) \ar[u]_{\Xi} \ar[r] & \ilim T_p H^1(G_S(K_n), E[p^{\infty}]) \ar[u]_{\Xi'} \ar[r]^-{\uppsi}
& \underset{i=1,2}{\bigoplus} \ilim T_p H^1(K_{\fp_{n,i}}, E)[p^{\infty}] \ar[u]_{\Xi''}
}
\end{equation}

To ease the notation, in the above diagram we have denoted $K_{n,{\fp_{n,i}}}$ by $K_{\fp_{n,i}}$. Applying the snake lemma to this diagram we get

$$0 \to \ker \Xi \to \ker \Xi' \to \ker \Xi'' \cap \img \uppsi \to \coker \Xi \to \coker \Xi'$$\\
From point (1) above, it follows that $\Xi'$ is an isomorphim i.e. $\ker \Xi'=0$ and $\coker \Xi'=0$. Therefore from the above sequence we get that $\ker \Xi=0$ as required. We also get that $\coker \Xi= \ker \Xi'' \cap \, \img \uppsi.$ But $\Xi''$ is the zero map so $\coker \Xi = \img \uppsi$.
Therefore we must show that $\rank_{\Lambda}(\img \uppsi) \le 2$. To study $\img \uppsi$ we use the Cassels-Poitou-Tate exact sequence (see \cite{CS}) which gives that the following sequence is exact

$$H^1(G_S(K_n), E[p^m]) \xrightarrow{\psi_{n,m}} \underset{v \in S_n}{\bigoplus} H^1(K_{n,v}, E)[p^m] \xrightarrow{\theta_{n,m}} \Selm_{p^m}(E/K_n)^{\dual}$$

Taking the points (3)-(5) above into consideration, we take the inverse limits of the groups over $m$ (using the multiplication by $p$ map) and then over $n$ (using the corestriction map). As the all the groups we are dealing with are compact Hausdorff, the resulting sequence is also exact:

$$\ilim T_p H^1(G_S(K_n), E[p^{\infty}]) \xrightarrow{\uppsi} \underset{i=1,2}{\bigoplus} \ilim T_p H^1(K_{\fp_{n,i}}, E)[p^{\infty}] \xrightarrow{\uptheta} \Selinf(E/K_{\infty})^{\dual}$$

The fact that this sequence is exact means that $\img \uppsi = \ker \uptheta$. So to show that $\rank_{\Lambda}(\img \uppsi) \le 2$ it suffices to show that $\rank_{\Lambda}(\ker \uptheta) \le 2$ or equivalently, if $\hat{\uptheta}$ is the dual map, that $\corank_{\Lambda}(\coker \hat{\uptheta}) \le 2$

By Tate local duality the dual of $H^1(K_{\fp_{n,i}}, E)[p^m]$ may be identified with $E(K_{\fp_{n,i}})/p^m$. Therefore using this fact, the map $\hat{\uptheta}$ becomes

$$\hat{\uptheta}: \Selinf(E/K_{\infty}) \to E(K_{\fp_{\infty, 1}}) \otimes \Qp/\Zp \times E(K_{\fp_{\infty, 2}}) \otimes \Qp/\Zp$$\\
This map is the usual map induced by restriction

$$H^1(K_{\infty}, E[p^{\infty}]) \to \underset{i=1,2}{\bigoplus} H^1(K_{\fp_{\infty,i}}, E[p^{\infty}])$$\\
Note that if $c \in \Selinf(E/K_{\infty}) \subset H^1(K_{\infty}, E[p^{\infty}])$ then its image under this map belongs to $E(K_{\fp_{\infty, 1}}) \otimes \Qp/\Zp \times E(K_{\fp_{\infty,2}}) \otimes \Qp/\Zp$

To prove our result we will first calculate $\corank_{\Lambda}(E(K_{\fp_{\infty,i}}) \otimes \Qp/\Zp)$. First we show that $E(K_{\fp{\infty,i}})[p^{\infty}]=\{0\}$. Since $\Gamma=\Gal(K_{\fp_{\infty,i}}/\Qp)$ is pro-$p$ it suffices to show that $E(\Qp)[p^{\infty}]=E(K_{\fp_{\infty,i}})[p^{\infty}]^{\Gamma}=\{0\}$. But since $E$ has supersingular reduction at $p$, we have $E(\Qp)[p^{\infty}]=\hat{E}(p\Zp)[p^{\infty}]$ where $\hat{E}$ is the formal group of $E/\Qp$. The result then follows from the fact (\cite{Silverman} ch. 4 th. 6.1) that $\hat{E}(p\Zp)$ has no $p$-torsion if $p \geq 3$.

Since $E(K_{\fp_{\infty,i}})[p^{\infty}]=\{0\}$, therefore, as in point (1) above, the restriction map induces an isomorphism $H^1(K_{\fp_{n,i}}, E[p^{\infty}]) \isomarrow H^1(K_{\fp_{\infty, i}}, E[p^{\infty}])^{\Gamma_n}$. In addition, we know that (\cite{Gb_IIC} ch.2) $\corank_{\Zp}(H^1(K_{\fp_{n,i}}, E[p^{\infty}]))=2p^n$. Therefore it follows that $\corank_{\Lambda}(H^1(K_{\fp_{\infty,i}}, E[p^{\infty}]))=2$.
But by point (2) above $E(K_{\fp_{\infty,i}}) \otimes \Qp/\Zp$ is isomorphic to $H^1(K_{\fp_{\infty,i}}, E[p^{\infty}])$ so we also have $\corank_{\Lambda}(E(K_{\fp_{\infty,i}}) \otimes \Qp/\Zp)=2$

We have now shown that

$$\corank_{\Lambda}(E(K_{\fp_{\infty, 1}}) \otimes \Qp/\Zp \times E(K_{\fp_{\infty, 2}}) \otimes \Qp/\Zp)=4$$\\
Therefore to show that $\corank_{\Lambda}(\coker \hat{\uptheta}) \le 2$ we only need to show that $\corank_{\Lambda}(\img \hat{\uptheta}) \geq 2$.This follows from a result of Ciperiani (\cite{Ciperiani} prop. 2.1): Consider the subgroup $E(K_{\infty}) \otimes \Qp/\Zp \subset \Selinf(E/K_{\infty})$.
Ciperiani shows that the image of the map (induced by restriction)

$$E(K_{\infty}) \otimes \Qp/\Zp \to E(K_{\fp_{\infty,i}}) \otimes \Qp/\Zp \qquad i=1,2$$\\
has $\Lambda$-corank greater than or equal to two. This implies the result.

\end{proof}

This control theorem implies the following key result

\begin{theorem}\label{Iwasawa_rank_main_theorem}
Both $\Selinf(E/K_{\infty})^{\; \dual}$ and $X_{p^{\infty}}(E/K_{\infty})$ are finitely generated $\Lambda$-modules
\begin{enumerate}[(a)]
\item If $E$ has good ordinary reduction at $p$, then $X_{p^{\infty}}(E/K_{\infty})$ is a free $\Lambda$-module and $\corank_{\Lambda}(\Selinf(E/K_{\infty})) = \rank_{\Lambda}(X_{p^{\infty}}(E/K_{\infty}))$

\item If $E$ has good supersingular reduction at $p$ and $p$ splits in $K/\Q$, then $\corank_{\Lambda}(\Selinf(E/K_{\infty})) \le \rank_{\Lambda}(X_{p^{\infty}}(E/K_{\infty}))+2$
\end{enumerate}
\end{theorem}

\begin{proof}

$\Selinf(E/K_{\infty})^{\; \dual}$ is a fintely generated $\Lambda$-module by \cite{Manin} th. 4.5. Therefore by \cite{NSW} prop. 5.5.10 we have that $Y_{p^{\infty}}(E/K_{\infty})$ is a finitely generated free $\Lambda$-module with the same rank as $\Selinf(E/K_{\infty})^{\; \dual}$

The control theorem above gives in both the ordinary and supersingular case an injection $X_{p^{\infty}}(E/K_{\infty}) \hookrightarrow Y_{p^{\infty}}(E/K_{\infty})$. Therefore $X_{p^{\infty}}(E/K_{\infty})$ is a finitely generated $\Lambda$-module. The other statements now follow from the control theorem.
\end{proof}

\section{Proof of Theorem A}
In this section we prove theorem A. We assume throughout this section the assumptions for theorem A in section 2.1. Let $\pi: X_0(N) \to E$ be the modular parametrization of section 2.2 and let $\pi_*: J_0(N) \to E$ be the corresponding covariant map. Before beginning the proof of theorem A we note that to prove theorem A we many assume that $\ker(\pi_*)$ is geometrically connected, for $E$ is $\Q$-isogenous to a strong Weil curve $E'$ having a modular parametrization with this property. Since $E$ is isogenous to $E'$, therefore if $E$ satisfies the conditions of theorem A so does $E'$. Moreover let $f: E \to E'$ and $g: E' \to E$ be a $\Q$-isogeny and its dual. Then $f \circ g = m$ and $g \circ f = m$ for some integer $m$. The isogenies $f$ and $g$ induce maps $\bar{g}: \Selinf(E/K_{\infty})^{\dual} \to \Selinf(E'/K_{\infty})^{\dual}$ and $\bar{f}: \Selinf(E'/K_{\infty})^{\dual} \to \Selinf(E/K_{\infty})^{\dual}$ whose composites are multiplication by $m$. From this we get that $\ker \bar{f}$ is annihilated by $m$ and therefore $\corank_{\Lambda}(\Selinf(E'/K_{\infty})) \le \corank_{\Lambda}(\Selinf(E/K_{\infty}))$. We get the reverse inequality from the map $\bar{g}$ and therefore we get equality. Similarly one can show that $\rank_{\Lambda}(X_{p^{\infty}}(E/K_{\infty})) = \rank_{\Lambda}(X_{p^{\infty}}(E'/K_{\infty}))$. This shows that we may (and will) indeed assume that $\ker(\pi_*)$ is geometrically connected.

Recall from section 2.2 that we have $R_n \alpha_n \subset R_{n+1} \alpha_{n+1}$. This allows us to take the direct limit $\dlim R_n \alpha_n \subset E(K_{\infty})/p$. We begin this section with the following important theorem

\begin{theorem}\label{Heegner_module_rank_theorem}
As a $\overbar{\Lambda}$-module $(\dlim R_n \alpha_n)^{\dual}$ is finitely generated and not torsion
\end{theorem}
\begin{proof}
It is well-known (see for example \cite{Manin} th. 4.5) that $\Selinf(E/K_{\infty})^{\dual}$ is a finitely generated $\Lambda$-module. Since $E(K_{\infty})[p^{\infty}]=\{0\}$ by corollary \ref{Ep_corollary} therefore we have an isomorphism $\Sel(E/K_{\infty}) \isomarrow \Selinf(E/K_{\infty})[p]$ and so $\Sel(E/K_{\infty})^{\dual}$ is a finitely generated $\overbar{\Lambda}$-module. The same then holds for $(\dlim R_n \alpha_n)^{\dual}$ (since it is a quotient of $\Sel(E/K_{\infty})^{\dual}$)

We now prove that $(\dlim R_n \alpha_n)^{\dual}$ is not $\overbar{\Lambda}$-torsion. Since finitely generated torsion $\overbar{\Lambda}$-modules are finite, we just have to show that $\dlim R_n \alpha_n$ is an infinite dimensional $\Fp$-vector space. Our result follows from a theorem of Cornut
(\cite{Cornut} th. B). Cornut defines a certain subgroup subgroup $M \subset \dlim R_n \alpha_n$. His theorem states that if $p$ is a prime not dividing $\varphi(Nd_K)$ nor the number of geoemetrically connected components of $\ker(\pi_*)$ then $\dim_{\Fp} M$ is infinite (which then implies that $\dim_{\Fp} \dlim R_n \alpha_n$ is infinite). Since $\ker(\pi_*)$ is geometrically connected we get the desired result.

\end{proof}

We would now like to show that there exists a Kolyvagin prime $\ell_1$ such that $\dlim \res_{\ell_1} R_n \alpha_n$ has nontrivial $\overbar{\Lambda}$-corank. To do this, we use the technique in \cite{CW}.

The above theorem implies that there exists a nonzero map

$$\phi: \overbar{\Lambda}^{\; \dual} \to \dlim R_n \alpha_n$$\\
Since $\phi^{\tau} -\phi$ and $\phi^{\tau} +\phi$ cannot be zero simultaneously, we can assume that $\phi$ lies in one of the eigenspaces for the action of complex multiplication $\tau$. Since $(\img \phi)^{\dual}$ injects into $\overbar{\Lambda}$, it is free of rank 1 over $\overbar{\Lambda}$. It follows that $\dim_{\Fp}(\img \phi)^{\Gamma}=1$. Also since $\img \phi$ is $\tau$-invariant and $\tau g \tau = g^{-1}$ for any $g \in \Gamma$ we get that $(\img \phi)^{\Gamma}$ is also $\tau$-invariant.

Now let $s \in (\img \phi)^{\Gamma} - \{0\}$. Then $s$ is an eigenvector for the action of $\tau$ on $\dlim R_n \alpha_n$. Since $E(K_{\infty})[p^{\infty}]=\{0\}$ (corollary \ref{Ep_corollary}) therefore the restriction map is an isomorphism

$$H^1(K, E[p]) \isomarrow H^1(K_{\infty}, E[p])^{\Gamma}$$\\
This implies that $s \in H^1(K, E[p])$. Now let $T$ be the subgroup generated by $s$ in $H^1(K, E[p])$. With the notation following proposition \ref{res_isom_prop} we have an extension $L_T/\Q$ which is Galois over $\Q$ since $T$ is $\tau$-invariant. Now let $H=\Gal(L_T/L) \cong E[p]$ (see \cite{Gross} prop. 9.3) and choose $h \in H$ such that $(\tau h)^2 \in H^+ - \{0\}$. We now choose an auxilary prime $\ell_1$ such that $\ell_1$ is relatively prime to $pNd_K$ and $\Frob_{\ell_1}(L_T/\Q)=[\tau h]$ (such a prime exists by the Chebotarev density theorem).

We now claim that $\res_{\ell_1} s \neq 0$. To prove this we only have to note that $\ell_1$ is inert in $K/\Q$ and hence $\Frob_{\ell_1}(L_T/K)=[(\tau h)^2]$. Since $(\tau h)^2$ is nonzero, the fact that $\res_{\ell_1} s \neq 0$ follows easily from the non-degeneracy of the pairing (\ref{pairing}). From this we get the following proposition.

\begin{proposition}\label{Kolyvagin_classes_rank_proposition}
As a $\overbar{\Lambda}$-module $(\dlim R_n c_n(\ell_1))^{\dual}$ is finitely generated and not torsion
\end{proposition}
\begin{proof}
To prove that $(\dlim R_n c_n(\ell_1))^{\dual}$ is a finitely generated $\overbar{\Lambda}$-module we cannot argue as in theorem \ref{Heegner_module_rank_theorem} because $\dlim R_n c_n(\ell_1)$ does not belong to $\Sel(E/K_{\infty})$. However it does belong to a "generalized" Selmer group which we will now define.

If $L/\Q$ is an algebraic extension and $T$ is any set of primes of $L$ we define $\Selm_p^T(E/L)$ by the exact sequence

$$\displaystyle 0 \longrightarrow \Selm_p^T(E/L) \longrightarrow H^1(L, E[p])\longrightarrow \prod_{v \notin T} H^1(L_v, E)[p]$$

We also define $\Selm_{p^{\infty}}^T(E/L)$ in a similar fashion. Now let $T$ be the set of primes of $K_{\infty}$ that lie above $\ell_1$. Then by property (2) of the Kolyvagin classes in section 2.2 we have that $\dlim R_n c_n(\ell_1) \subset \Selm_p^T(E/K_{\infty})$. We now modify the argument in theorem \ref{Heegner_module_rank_theorem}: Manin (\cite{Manin} th. 4.5) has shown that $\Selm_{p^{\infty}}^T(E/K_{\infty})^{\dual}$ is a finitely generated $\Lambda$-module (he proves this for any set $T$). Since $E(K_{\infty})[p^{\infty}]=\{0\}$ (by corollary \ref{Ep_corollary}) therefore we have an isomorphism $\Selm_p^T(E/K_{\infty}) \isomarrow \Selm_{p^{\infty}}^T(E/K_{\infty})$ and so $\Selm_p^T(E/K_{\infty})^{\dual}$ is a finitely generated $\overbar{\Lambda}$-module. Then the same holds for $(\dlim R_n c_n(\ell_1))^{\dual}$ (since it is a quotient of $\Selm_p^T(E/K_{\infty})^{\dual}$).

We now prove that $(\dlim R_n c_n(\ell_1))^{\dual}$ is not a torsion $\overbar{\Lambda}$-module. Recall that we have a nonzero map

$$\phi: \overbar{\Lambda}^{\; \dual} \to \dlim R_n \alpha_n$$\\
We have chosen $s \in \img \phi$ and chosen the prime $\ell_1$ so that $\res_{\ell_1} s \neq 0$. Now consider the restriction map

$$\psi: \dlim R_n \alpha_n \to \dlim \res_{\ell_1} R_n \alpha_n$$\\
Since we have $\res_{\ell_1} s \neq 0$ therefore $\img(\psi \circ \phi) \neq 0$ and since $\overbar{\Lambda}^{\; \dual}$ surjects onto $\img(\psi \circ \phi)$ therefore $\img(\psi \circ \phi)^{\dual}$ is a nonzero submodule of $\overbar{\Lambda}$ and hence free of rank 1 over $\overbar{\Lambda}$. It follows that $\dlim \res_{\ell_1} R_n \alpha_n$ is not $\overbar{\Lambda}$-cotorsion

Since $\dlim R_n c_n(\ell_1)$ surjects onto $\dlim R_n d_n(\ell_1)$ and since by property (3) of the Kolyvagin classes we have an isomorphism $\dlim \res_{\ell_1} R_n d_n(\ell_1) \cong \dlim \res_{\ell_1} R_n \alpha_n$ therefore $\dlim R_n c_n(\ell_1)$ surjects onto $\dlim \res_{\ell_1} R_n \alpha_n$. It follows that $\dlim R_n c_n(\ell_1)$ is not $\overbar{\Lambda}$-cotorsion since $\dlim \res_{\ell_1} R_n \alpha_n$ is not $\overbar{\Lambda}$-cotorsion
\end{proof}

The above proposition implies that there exists a nonzero map

$$\phi': \overbar{\Lambda}^{\; \dual} \to \dlim R_n c_n(\ell_1)$$\\
As with the map $\phi$ defined earlier we may assume that $\phi'$ lies in an eigenspace for the action of complex multiplication $\tau$. Now let $s' \in (\img \phi')^{\Gamma} -\{0\}$. Then $s'$ is an eigenvector for the action of $\tau$ on $\dlim R_n c_n(\ell_1)$

We now have the following proposition

\begin{proposition}\label{linear_independence_proposition}
$s$ and $s'$ are linearly independent over $\Fp$
\end{proposition}
\begin{proof}
$s \in R_n \alpha_n$ for some $n$ and $s' \in R_m c_m(\ell_1)$ for some $m$. We may assume that $n=m$. To ease the notation we will denote $H^1((K_n[\ell_1])_{\ell_1}, E[p])$ by $H^1(K_{n,\ell_1}[\ell_1], E[p])$

Now let $\psi$ be the composition of the restriction maps

$$\psi: H^1(K_n, E[p]) \to H^1(K_{n, \ell_1}, E[p]) \to H^1(K_{n, \ell_1}[\ell_1], E[p])$$\\
We claim that $\psi(s) \neq 0$. Consider the field $L_T$ defined before proposition \ref{Kolyvagin_classes_rank_proposition} and let $\lambda$ be a prime of $K[\ell_1]L_T$ above $\ell_1$. It is not difficult to see that to show $\psi(s) \neq 0$ we only need to prove that the completions of the extensions $L_T/L$ and $LK[\ell_1]/L$ at $\lambda$ are disjoint. We know that $K[\ell_1]/K[1]$ is totally ramified at any prime at any prime above $\ell_1$. Since $\ell_1$ does not divide $Np$ it is therefore unramified in $L/\Q$. It follows that $LK[\ell_1]/LK[1]$ is totally ramified at primes above $\ell_1$. Since $L_T/L$ is unramified at primes above $\ell_1$ therefore to show that the completions of the extensions $L_T/L$ and $LK[\ell_1]/L$ at $\lambda$ are disjoint we only need to show that the completions of $LK[1]/L$ and $L_T/L$ are disjoint. But this follows from the fact that $\Gal(L_T/L)\cong E[p]$ (see \cite{Gross} prop. 9.3) and that the class number of $K$ is prime to $p$.

If $\psi(s')=0$ then $s$ and $s'$ must be linearly independent. So now we consider the case where $\psi(s') \neq 0$. We have $s'=r c_n(\ell_1)$ for some $r \in R_n$. Since $\psi(s') \neq 0$ therefore from the commutative diagram (\ref{kolyvagin_classes_diagram}) in section 2.2 we get that $D_{\ell_1} (\res_{\ell_1} (\text{Tr}_{K_n[1]/K_n} r \alpha_n(\ell_1))=\res_{\ell_1}(\text{Tr}_{K_n[1]/K_n} D_{\ell_1} r \alpha_n(\ell_1)) \neq 0$ and so $\res_{\ell_1} (\text{Tr}_{K_n[1]/K_n} r \alpha_n(\ell_1)) \neq 0$.(Note that the fields $K_n$ and $K[\ell_1]$ are linearly disjoint over $K$. Therefore $\Gal(K_n[\ell_1]/K)=\Gal(K_n/K) \times \Gal(K[\ell_1]/K)$ and so $r \alpha_n(\ell_1)$ makes sense).

For any prime $\lambda$ of $K_n[\ell_1]$ we will denote the residue field by $\mathbf{K}_{\lambda}$. Now define $E(\mathbf{K}_{\ell_1})/p:=\oplus_{\lambda | \ell_1} \tilde{E}(\mathbf{K}_{\lambda})/p$ where $\tilde{E}$ is the reduced elliptic curve and the sum is taken over all primes $\lambda$ dividing $\ell_1$

With this notation we let $\red_{\ell_1} : E(K_{n,\ell_1}[\ell_1])/p \to E(\mathbf{K}_{\ell_1})/p$ be the reduction map and we let $\overbar{\res}_{\ell_1}: E(K_n[\ell_1])/p \to E(\mathbf{K}_{\ell_1})/p$ be the composition $\overbar{\res}_{\ell_1} =\red_{\ell_1} \circ \res_{\ell_1}$

It follows from proposition \ref{Eichler-Shimura_proposition} that we have

$$\overbar{\res}_{\ell_1}(\text{Tr}_{K_n[1]/K_n} r \alpha_n(\ell_1)) =\Frob_{\ell_1}\overbar{\res}_{\ell_1}(\text{Tr}_{K_n[1]/K_n} r \alpha_n(1)) =\Frob_{\ell_1}\overbar{\res}_{\ell_1}(r \alpha_n)$$\\
Since $\ell_1 \neq p$ therefore multiplication by $p$ is an isomorphism on the formal group of $E(K_{n, \lambda}[\ell_1])$ for any $\lambda$ above $\ell_1$ ($K_{n, \lambda}[\ell_1]$ denotes the completion of $K_n[\ell_1]$ at $\lambda$). From this it follows that the map $\red_{\ell_1}$ is an isomorphism and since we have shown earlier that $\res_{\ell_1}(\text{Tr}_{K_n[1]/K_n} r \alpha_n(\ell_1)) \neq 0$ therefore it follows that $\overbar{\res}_{\ell_1}(\text{Tr}_{K_n[1]/K_n} r \alpha_n(\ell_1)) \neq 0$ and hence from what we deduced above we get $\overbar{\res}_{\ell_1}(r \alpha_n) \neq 0$. We conclude that $\res_{\ell_1}(r \alpha_n) \neq 0$

Using property (3) of the Kolyvagin classes in section 2.2, it follows from this that $\res_{\ell_1}(r d_n(\ell_1)) \neq 0$ and so $r d_n(\ell_1) \neq 0$. Now consider the exact sequence

$$ 0 \rightarrow E(K_n)/p \rightarrow H^1(K_n, E[p]) \xrightarrow{\varphi} H^1(K_n, E)[p] \rightarrow 0$$\\
If $\varphi$ is the above map then $\varphi(s')= r d_n(\ell_1) \neq 0$. But $s \in E(K_n)/p$ so from the above exact sequence we get that $\varphi(s) =0$. This proves that $s$ and $s'$ are linearly independent.
\end{proof}

We are now ready to define the subgroup $S_{n_0} \subset H^1(K_{n_0}, E[p])$ and the set $U$ in section 2.3. We consider 2 cases:\\

\noindent \textit{Case 1.} $s$ and $s'$ lie in different eigenspaces for the action of complex conjugation $\tau$\\

In this case if $\tau s = \epsilon \, s$ where $\epsilon \in  \{+1, -1\}$ then $\tau s' = -\epsilon \, s'$. Note that $s$ and $s'$ are linearly independent over $\Fp$. Since $E(K_{\infty})[p^{\infty}]=\{0\}$ by corollary \ref{Ep_corollary} therefore the restriction map is an isomorphism:

$$H^1(K, E[p]) \isomarrow H^1(K_{\infty}, E[p])^{\Gamma}$$\\
Since $s$ and $s'$ are both $\Gamma$-invariant therefore they belong to $H^1(K, E[p])$. Now let $S$ be the subgroup of $H^1(K, E[p])$ generated by $s$ and $s'$. We now let $n_0=0$ and $S_{n_0}=S$. Let $V=\Gal(L_S/L)$ where $L=K(E[p])$. We will denote $L_{\{\Fp s\}}$ and $L_{\{\Fp s'\}}$ by $L_s$  and $L_{s'}$ respectively. By \cite{Gross} prop. 9.3 we have

$$V=\Gal(L_s/L) \times \Gal(L_{s'}/L) =E[p] \times E[p]$$\\
Complex conjugation $\tau$ acts on $V$ by

$$\tau(x,y)\tau = (\epsilon \, \tau x, -\epsilon \, \tau y)$$\\
Let $E[p]^{\epsilon}$ denote the submodule of $E[p]$ on which $\tau$ acts as $\epsilon$. We now define a subset $U$ of $V$ as

$$U=\{(x,y) \; | \; x \in E[p]^{\epsilon}-\{0\} \ \text{and} \ y \in E[p]^{-\epsilon} - \{0\}\}$$\\
It is clear that $U^+$ generates $V^+$\\

\noindent \textit{Case 2.} $s$ and $s'$ lie in the same eigenspace for the action of complex conjugation $\tau$\\

In this case we have $\tau s = \epsilon \, s$ and $\tau s' = \epsilon \, s'$ for some $\epsilon \in \{ +1, -1\}$. Following \cite{CW} we now consider the $\Gamma$-invariants of $\img \phi/\langle s \rangle$ and $\img \phi'/ \langle s' \rangle$. The map $\phi$ induces a surjection:

$$\overbar{\phi}: \overbar{\Lambda}^{\; \dual} \twoheadrightarrow \img \phi/\langle s \rangle$$\\
Therefore $(\img \phi/ \langle s \rangle)^{\; \dual}$ is a nonzero submodule of $\overbar{\Lambda}$ and so is free of rank 1 over $\overbar{\Lambda}$. It follows that the $\Gamma$-invariants of $\img \phi/\langle s\rangle$ is a 1-dimensional $\Fp$-vector space. Moreover, since $\tau$ acts on $\img \phi/\langle s \rangle$ and $\tau g \tau =g^{-1}$ for any $g \in \Gamma$ it follows that the $\Gamma$-invariants of $\img \phi/\langle s \rangle$ are $\tau$-invariant.

Now choose $e \in \img \phi$ such that $e + \langle s \rangle \in (\img \phi/\langle s \rangle)^{\Gamma} -\{0\}$. Then $e + \langle s \rangle$ is an eigenvector for the action of $\tau$ on $\img \phi/\langle s \rangle$. So $\tau e = \epsilon' e +xs$ for some $\epsilon' \in \{+1, -1\}$ and $x \in \Fp$. As in \cite{CW}, we will now show that $\epsilon'=-\epsilon$. We have

$$e = \epsilon' \tau e + x\tau s = e + \epsilon' x s + \epsilon \, xs = e +(\epsilon' + \epsilon)xs$$\\
Therefore $\epsilon' =-\epsilon$ if $x \neq 0 $. So we still need to consider the case where $\tau e = \epsilon' e$. Here we use the fact that if $g$ is topological generator of $\Gamma$ then $(g-1)e= ys$ for some nonzero $y \in \Fp$. We have

$$ \tau (g-1) e = (g^{-1} - 1)\epsilon' e = -\epsilon'g^{-1}[(g-1)e]= -\epsilon'g^{-1}ys = -\epsilon'ys$$\\
On the other hand $\tau(g-1)e = y\tau s = \epsilon \, y s$ so $\epsilon' = -\epsilon$.

Now consider $e'= e - \epsilon \, \frac{1}{2} x s$. We have

$$\tau e' = \tau e -\epsilon \, \frac{1}{2} x \tau s = -\epsilon \, e + xs - (\epsilon)^2 \frac{1}{2} xs = -\epsilon \, e + \frac{1}{2} xs =
-\epsilon(e -  \epsilon \, \frac{1}{2} xs) = -\epsilon \, e'$$\\
Therefore replacing $e$ with $e'$ we have $\tau e = -\epsilon e$

To summarize, there exists $e \in \img \phi$ such that $e + \langle s \rangle \in (\img \phi/\langle s \rangle)^{\Gamma} -\{0\}$ and $\tau e = -\epsilon \, e$. Similarly there exists $e' \in \img \phi'$ such that $e' + \langle s' \rangle \in (\img \phi'/\langle s' \rangle)^{\Gamma} -\{0\}$ and $\tau e' = -\epsilon \, e'$

We now show that $s, s', e$ and $e'$ are linearly independent over $\Fp$. We will use the fact that $s$ and $s'$ are $\Gamma$-invariant and that $(g-1)e = y s$ and $(g-1)e' = y' s'$ for some nonzero $y, y' \in \Fp$. Now suppose that

$$k_1s + k_2 s' + k_3 e + k_4e' = 0$$\\
For some $k_i \in \Fp$. Multiplying both sides by $g-1$ yields

$$k_3ys + k_4y's' = 0$$\\
and therefore $k_3=k_4=0$ by proposition \ref{linear_independence_proposition}. So we have

$$k_1s +k_2s' = 0$$\\
So using proposition \ref{linear_independence_proposition} again gives $k_1=k_2=0$

Now choose $n \in \N$ such that $s, s', e$ and $e'$ all belong to $H^1(K_n, E[p])$ and let $S$ be the subgroup of $H^1(K_n, E[p])$ generated by $s, s', e$ and $e'$. Then $\dim_{\Fp} S = 4 $. Note that $S$ is stable under $\Gal(K_n/\Q)$. Now let $n_0=n$ and $S_{n_0} = S$. Let $V=\Gal(L_S/L)$ where $L = K_n(E[p])$. We will denote $L_{\{\Fp s\}}$ by $L_s$ and similarly for $s',e$ and $e'$. Since $\dim_{\Fp} S=4$ therefore by \cite{Gross} prop. 9.3 we have

$$V=\Gal(L_s/L) \times \Gal(L_{s'}/L) \times \Gal(L_e/L) \times \Gal(L_{e'}/L) \cong E[p]^4$$\\
Complex conjugation $\tau$ acts on $V$ by

$$\tau(x,y,z,w)\tau = (\epsilon \, \tau x, \epsilon \, \tau y, -\epsilon \, \tau z, -\epsilon \, \tau w)$$\\
Let $E[p]^{\epsilon}$ denote the submodule of $E[p]$ on which $\tau$ acts as $\epsilon$. We now define

$$U_1=\{(x,0, 0, z) \; | \; x \in E[p]^{\epsilon}-\{0\} \ \text{and} \ z \in E[p]^{-\epsilon}-\{0\} \}$$\\
and
$$U_2=\{(0,y, w, 0) \; | \; y \in E[p]^{\epsilon}-\{0\} \ \text{and} \ w \in E[p]^{-\epsilon}-\{0\} \}$$\\
Finally we let

$$U=U_1 \cup U_2$$\\
It is clear that $U^+$ generates $V^+$

For any $\ell \in \mathscr{L}(U)$ we consider $\dlim \res_{\ell} R_n \alpha_n$ and $\dlim \res_{\ell} R_n c_n(\ell_1)$. We have the following key proposition

\begin{proposition}
For any $\ell \in \mathscr{L}(U)$ the submodules $\dlim \res_{\ell} R_n \alpha_n$ and $\dlim \res_{\ell} R_n c_n(\ell_1)$ of $\dlim E(K_{n,\ell})/p$ each have $\overbar{\Lambda}$-corank greater or equal to 1 and together they generate a submodule of $\overbar{\Lambda}$-corank equal to 2
\end{proposition}
\begin{proof}
Consider the nonzero maps we defined earlier

$$\phi: \overbar{\Lambda}^{\; \dual} \to \dlim R_n \alpha_n$$

$$\phi': \overbar{\Lambda}^{\; \dual} \to \dlim R_n c_n(\ell_1)$$\\
For any $\ell \in \mathscr{L}(U)$ let

$$\psi_{\ell}: \dlim R_n \alpha_n \to \dlim \res_{\ell} R_n \alpha_n$$

$$\psi'_{\ell}: \dlim R_n c_n(\ell_1) \to \dlim \res_{\ell} R_n c_n(\ell_1)$$\\
be the restriction maps. Then for any $\ell \in \mathscr{L}(U)$ we define $\overbar{\phi}_{\ell} =\psi_{\ell} \circ \phi$ and $\overbar{\phi}'_{\ell}= \psi'_{\ell} \circ \phi'$. Our definition of the set $U$ shows that for any $\ell \in \mathscr{L}(U)$ there exist $\alpha \in \img \phi$ and $\beta \in \img \phi'$ such that $\psi_{\ell}(\alpha) \neq 0$ and $\psi'_{\ell}(\beta) \neq 0$. This shows that $\img \overbar{\phi}_{\ell} \neq 0$ and $\img \overbar{\phi}'_{\ell} \neq 0$. Since $\overbar{\Lambda}^{\; \dual}$ surjects onto $\img \overbar{\phi}_{\ell}$ and $\img \overbar{\phi}'_{\ell}$. Therefore both $\overbar{\phi}^{\; \dual}_{\ell}$ and $\overbar{\phi}'^{\; \dual}_{\ell}$ are nonzero submodules of $\overbar{\Lambda}$ and so are free of rank 1 over $\overbar{\Lambda}$. This shows that both $\dlim \res_{\ell} R_n \alpha_n$ and $\dlim \res_{\ell} R_n c_n(\ell_1)$ are not $\overbar{\Lambda}$-cotorsion

To show that the $\overbar{\Lambda}$-corank of $\dlim \res_{\ell} R_n \alpha_n + \dlim \res_{\ell} R_n c_n(\ell_1)$ is equal to 2, it suffices to prove that $\img \overbar{\phi}_{\ell} + \img \overbar{\phi}'_{\ell}$ has $\overbar{\Lambda}$-corank equal to 2, for then this implies that $\dlim \res_{\ell} R_n \alpha_n + \dlim \res_{\ell} R_n c_n(\ell_1)$ has $\overbar{\Lambda}$-corank greater or equal to 2. But by the isomorphism in property (3) of the Kolyvagin classes in section 2.2 together with proposition \ref{Iwasawa_rank_proposition} we have that the $\overbar{\Lambda}$-corank of $\dlim E(K_{n, \ell})/p$ is equal to 2 so we get equality.

We will show that $\img \overbar{\phi}_{\ell} + \img \overbar{\phi}'_{\ell}$ has $\overbar{\Lambda}$-corank equal to 2. To show this we only need to show that $\dim_{\Fp}(\img \overbar{\phi}_{\ell} + \img \overbar{\phi}'_{\ell})^{\Gamma} = 2$. To see this we note that the maps $\overbar{\phi}_{\ell}$ and $\overbar{\phi}'_{\ell}$ induce a surjection

$$(\overbar{\Lambda}\oplus \overbar{\Lambda})^{\; \dual} \twoheadrightarrow \img \overbar{\phi}_{\ell} + \img \overbar{\phi}'_{\ell}$$\\
Therefore $(\img \overbar{\phi}_{\ell} + \img \overbar{\phi}'_{\ell})^{\; \dual}$ is a nonzero submodule of $\overbar{\Lambda} \oplus \overbar{\Lambda}$ and hence free of rank 1 or 2 over $\overbar{\Lambda}$. Since $\dim_{\Fp}(\img \overbar{\phi}_{\ell} + \img \overbar{\phi}'_{\ell})^{\Gamma} = 2$ therefore we must have that $(\img \overbar{\phi}_{\ell} + \img \overbar{\phi}'_{\ell})^{\; \dual}$ has $\overbar{\Lambda}$-rank equal to 2.

We now show that $\dim_{\Fp}(\img \overbar{\phi}_{\ell} + \img \overbar{\phi}'_{\ell})^{\Gamma} =2$. Recall that we have chosen $s \in (\img \phi)^{\Gamma} - \{0\}$ and $s' \in (\img \phi')^{\Gamma} - \{0\}$. In case 1 $s$ and $s'$ belong to different eigenspaces for the action of complex conjugation $\tau$. Our definition of the set $U$ in this case gives that for any $\ell \in \mathscr{L}(U)$ we have $\res_{\ell}(s) \neq 0$ and $\res_{\ell}(s') \neq 0$. Also since $s$ and $s'$ are $\Gamma$-invariant and belong to different eigenspaces for the action of $\tau$, the same will be true for $\res_{\ell}(s)$ and $\res_{\ell}(s')$. Therefore we get the desired result in this case.

Now we consider case 2 where $s$ and $s'$ belong to the same eigenspace for the action of $\tau$. In this case we have chosen elements $e$ and $e'$ such that $e + \langle s \rangle \in (\img \phi / \langle s \rangle)^{\Gamma}$, $e' + \langle s \rangle \in (\img \phi' / \langle s' \rangle)^{\Gamma}$ and $e$ and $e'$ are eigenvectors for the action of $\tau$ belonging to a different eigenspace from $s$ and $s'$. We have in this case defined our set $U$ to be the union of 2 sets $U_1$ and $U_2$. Suppose that $\ell$ belongs to $\mathscr{L}(U_1)$. The definition of the set $U_1$ shows that $\res_{\ell}(s) \neq 0$, $\res_{\ell}(e') \neq 0$ and $\res_{\ell}(s') =0$. Since $\res_{\ell}(s') =0$ therefore $\res_{\ell}(e')$ is $\Gamma$-invariant. Then we get the desired result because $s$ and $e'$ belong to different eigenspaces for the action of $\tau$. We get a similar situation if $\ell$ belongs to $\mathscr{L}(U_2)$ thereby completing the proof.
\end{proof}

\begin{corollary}\label{Iwasawa_rank_corollary}
For any $\ell \in \mathscr{L}(U)$ the submodules $\dlim \res_{\ell} R_n d_n(\ell)$ and $\dlim \res_{\ell} R_n d_n(\ell \ell_1)$ of $\dlim H^1(K_{n, \ell}, E)[p]$ each have $\overbar{\Lambda}$-corank greater or equal to 1 and together they generate $\dlim H^1(K_{n, \ell}, E)[p]$
\end{corollary}
\begin{proof}
Using property (3) of the Kolyvagin classes in section 2.2, it follows from the previous proposition that both $\dlim \res_{\ell} R_n d_n(\ell)$ and $\dlim \res_{\ell} R_n d_n(\ell \ell_1)$ have $\overbar{\Lambda}$-corank greater or equal to 1 and that together they generate a submodule of $\dlim H^1(K_{n, \ell}, E)[p]$ of $\overbar{\Lambda}$-corank equal to 2. This submodule must be equal to $\dlim H^1(K_{n, \ell}, E)[p]$ since by proposition \ref{Iwasawa_rank_proposition} $H^1(K_{n, \ell}, E)[p]$ is a cofree $\overbar{\Lambda}$-module of rank 2.
\end{proof}

\begin{proposition}
For any $\ell \in \mathscr{L}(U)$, $\img \uppsi_{\ell}$ is a cofree $\overbar{\Lambda}$-module and $\img \uppsi_{\ell}=\uppsi_{\ell}(\dlim \res_{\ell} R_n d_n(\ell \ell_1))$
\end{proposition}
\begin{proof}
The fact that $\img \uppsi(\ell)$ is a cofree $\overbar{\Lambda}$-module follows from the fact that $\dlim H^1(K_{n,\ell}, E)[p]$ is cofree (proposition \ref{Iwasawa_rank_proposition}). As for the fact that $\img \uppsi_{\ell}=\uppsi_{\ell}(\dlim \res_{\ell} R_n d_n(\ell \ell_1))$, it follows from corollary \ref{Iwasawa_rank_corollary} together with the fact that by proposition \ref{global_duality_proposition} and property (2) of the Kolyvagin classes in section 2.2 we get that $\uppsi_{\ell}(\dlim \res_{\ell} R_n d_n(\ell)) = 0$.
\end{proof}

\begin{proposition}
We have $\rank_{\overbar{\Lambda}}(X_p(E/K_{\infty})) \le 1$
\end{proposition}
\begin{proof}
According to proposition \ref{generating_Xp_proposition} $\img \uppsi_{\ell}$ with $\ell$ ranging over $\mathscr{L}(U)$ generate $X_p(E/K_{\infty})^{\; \dual}$. Write $\mathscr{L}$ as a disjoint union $\mathscr{L} = \mathscr{L}_1 \dotcup \mathscr{L}_2$ where $\mathscr{L}_1$ (resp. $\mathscr{L}_2$) consists of the primes $\ell \in \mathscr{L}$ such that $\uppsi_{\ell}(\dlim \res_{\ell} R_n d_n(\ell \ell_1))$ is zero (resp. nonzero).

If $\mathscr{L}_2$ is empty, then $X_p(E/K_{\infty})=\{0\}$. Otherwise assume that $\mathscr{L}_2$ is nonempty. We will now show that $\rank_{\overbar{\Lambda}}(X_p(E/K_{\infty})) =1$

Recall that $\ell_1$ was chosen so that $\corank_{\overbar{\Lambda}}(\dlim \res_{\ell_1} R_n \alpha_n) \geq 1$. Therefore by property (3) of the Kolyvagin classes in section 2.2 we have $\corank_{\overbar{\Lambda}}(\dlim \res_{\ell_1} R_n d_n(\ell_1)) \geq 1$. Just as in the previous proposition we have that $\uppsi_{\ell_1}(\dlim \res_{\ell_1} R_n d_n(\ell_1)) = 0$ and that $\img \uppsi_{\ell_1}$ is a cofree $\overbar{\Lambda}$-module. Since $\corank_{\overbar{\Lambda}}(\dlim \res_{\ell_1} R_n d_n(\ell_1)) \geq 1$ therefore taking proposition \ref{Iwasawa_rank_proposition} into account it follows that $\img \uppsi_{\ell_1}$ is a cofree $\overbar{\Lambda}$-module of rank less than or equal to 1.

Now assume that $\ell \in \mathscr{L}_2$ and $\alpha \in \dlim R_n d_n(\ell \ell_1)$. Then using proposition \ref{global_duality_proposition} together with property (2) of the Kolyvagin classes in section 2.2 it follows that

$$\uppsi_{\ell}(\res_{\ell}(\alpha)) + \uppsi_{\ell_1}(\res_{\ell_1}(\alpha)) = 0$$\\
So $\uppsi_{\ell}(\res_{\ell}(\alpha)) \in \img \uppsi_{\ell_1}$ and so by the previous proposition we get $\img \uppsi_{\ell} \subseteq \img \uppsi_{\ell_1}$. Then using our earlier obervation that $\img \uppsi_{\ell_1}$ is $\overbar{\Lambda}$-cofree of rank less than or equal to 1 together with the facts that $\img \uppsi_{\ell}$ is nonzero and cofree as a $\overbar{\Lambda}$-module (from the previous propostion) it follows that $\img \uppsi_{\ell}=\img \uppsi_{\ell_1}$. Our desired result then follows from this.
\end{proof}

We can now finally prove theorem A

\begin{theoremA}
Assume that $E$ has good ordinary reduction at $p$, then $\Selinf(E/K_{\infty})$ has $\Lambda$-corank equal to 1 and $X_{p^\infty}(E/K_{\infty})$ is a free $\Lambda$-module of rank 1.
\end{theoremA}\begin{proof}

From theorem \ref{Iwasawa_rank_main_theorem} together with the previous proposition we get that $\Selinf(E/K_{\infty})^{\; \dual}$ is a finitely generated $\Lambda$-module, $X_{p^{\infty}}(E/K_{\infty})$ is a finitely generated free $\Lambda$-module and $X_p(E/K_{\infty})$ is a finitely generated free $\overbar{\Lambda}$-module with

$$\corank_{\Lambda}(\Selinf(E/K_{\infty})) = \rank_{\Lambda}(X_{p^{\infty}}(E/K_{\infty})) = \rank_{\overbar{\Lambda}}(X_p(E/K_{\infty})) \le 1 $$\\
Therefore the theorem will follow if we can show that $\corank_{\Lambda}(\Selinf(E/K_{\infty})) \geq 1$. We now show this as follows: using theorem 1.4 of \cite{Vatsal} together with the main result of \cite{BD} it follows that $\rank(E(K_n))=p^n +O(1)$ and therefore $\corank_{\Zp}(E(K_n) \otimes \Qp/\Zp)= p^n +O(1)$. This implies that $\corank_{\Lambda}(E(K_{\infty}) \otimes \Qp/\Zp) \geq 1$ and since $E(K_{\infty}) \otimes \Qp/\Zp$ is contained in $\Selinf(E/K_{\infty})$ it therefore follows that $\corank_{\Lambda}(\Selinf(E/K_{\infty})) \geq 1$ as desired.
\end{proof}

\section{Proof of Theorem B}

In this section we prove theorem B. We assume thoroughout this section the assumptions for theorem B in section 2.1. By the same argument in the beginning of section 3 we may (and will) assume that $E$ is a strong Weil curve such that the map $\pi_*: J_0(N) \to E$ coming from the modular parametrization $\pi: X_0(N) \to E$ has a geometrically connected kernel.

Recall from section 2.2 that we have $R_n \alpha_n \subset R_{n+2} \alpha_{n+2}$ for any $n$. This allows us to construct the direct limits $\dlim R_{2n} \alpha_{2n}$ and $\dlim R_{2n+1} \alpha_{2n+1}$. We recall the following conjecture from the introduction

\begin{conjecture*}
If $p$ splits in $K/\Q$ and $E$ has good supersingular reduction at $p$ then the submodule of $E(K_{\infty})/p$ generated by $\dlim R_{2n} \alpha_{2n}$ and $\dlim R_{2n+1} \alpha_{2n+1}$ has $\overbar{\Lambda}$-corank greater than or equal to two.
\end{conjecture*}

We will now discuss some evidence for this conjecture. The main evidence comes from the following theorem. See also the remarks after the theorem.

\begin{theorem}\label{Heegner_modules_rank_theorem}
The $\overbar{\Lambda}$-modules $\dlim R_{2n} \alpha_{2n}$ and $\dlim R_{2n+1} \alpha_{2n+1}$ are not cotorsion
\end{theorem}

\begin{proof}
The fact that $(\dlim R_{2n} \alpha_{2n})^{\; \dual}$ and $(\dlim R_{2n+1} \alpha_{2n+1})^{\; \dual}$ are finitely generated over $\overbar{\Lambda}$ follows exactly as in theorem \ref{Heegner_module_rank_theorem}. Since finitely generated torsion $\overbar{\Lambda}$-modules are finite, therefore to show that $\dlim R_{2n} \alpha_{2n}$ and $\dlim R_{2n+1} \alpha_{2n+1}$ are not cotorsion we only have to show that these modules are infinite.

We will show this by the same method used to prove theorem B in Cornut's paper \cite{Cornut}. Using the same notation as in Cornut's paper we let $N$ denote the conductor of $E$ and we let $M$ be the integer defined on page 517 of Cornut's paper. Recall from section 2.2 that $K[p^{\infty}]=\cup_{n \geq 1} K[p^n]$. Now let $\ell$ be a rational prime not dividing $2NMp$ that is inert in $K/\Q$. Choose a place $v_{\ell}$ of $K[p^{\infty}]$ above $\ell$ and let $k(\ell)$ denote its residue field, so that $k(\ell) \cong \mathbb{F}_{\ell^2}$. We have a reduction map at $v_{\ell}$

$$\text{red}_{\ell} : X_0(NM)(K[p^{\infty}]) \to X_0(NM)(k(\ell))$$\\
As Cornut notes, $\text{red}_{\ell}$ maps any CM point (relative to $K$) to the supersingular locus $X^{ss}_0(NM)(k(\ell))$ of $X_0(NM)(k(\ell))$.

Now let $\mathscr{L}_p$ be the set defined on page 506 of Cornut's paper. Using the notation on page 517 of the paper we have for any $a \in \mathscr{L}_p$ a Heegner point $H'(a) \in X_0(NM)(K[p^{\infty}])$. Let $c(a)$ denote the conductor of $H'(a)$. We now define the following subsets of $X_0(NM)(K[p^{\infty}])$

$$X^+ = \{ \sigma(H'(a)) \; | \; a \in \mathscr{L}_p, \; \sigma \in \Gal(K[p^{\infty}]/K) \; \text{and} \; \text{ord}_p(c(a)) \equiv 0 \mod 2\}$$

$$X^- = \{ \sigma(H'(a)) \; | \; a \in \mathscr{L}_p, \; \sigma \in \Gal(K[p^{\infty}]/K) \; \text{and} \; \text{ord}_p(c(a)) \equiv 1 \mod 2\}$$\\
We also define $X$ be the set of all Heegner points relative to $K$ of $p$-power conductor.

Let $\widehat{K} = K \otimes \, \widehat{\Z}$ and $\widehat{K}^{(p)} = K \otimes \, \widehat{\Z}^{(p)}$ where $\widehat{\Z}$ and $\widehat{\Z}^{(p)}$ are the profinite and ``prime to p-adic''completions of $\Z$. Denote the Artin reciprocity map as

$$[K[p^{\infty}]/K, \star]: \widehat{K}^* \to \Gal(K[p^{\infty}]/K)$$\\
Now let $S$ be a finite set of rational primes $\ell \nmid 2NMp$ which are inert in $K/\Q$ and let $\mathcal{R}$ be a finite subset of $\Gal(K[p^{\infty}]/K)$ consisting of elements that are pairwise distinct modulo $[K[p^{\infty}]/K, \widehat{K}^{(p)*}]$. For any such sets $S$ and $\mathcal{R}$ we define the following map

\begin{align*}
\psi: \, & X \to \prod_{\ell \in S} X^{ss}_0(NM)(k(\ell))^{\mathcal{R}}\\
& \alpha \mapsto (\text{red}_{\ell}(\sigma(\alpha)))_{\sigma \in \mathcal{R}, \ell \in S}
\end{align*}\\
From Cornut's proof of theorem B in his paper, we see that to obtain our desired result we only have to show for any sets $S$ and $\mathcal{R}$ as above that $\psi|_{X^+}$ and $\psi|_{X^-}$ are surjective.

To show this last statement we use the work of Cornut and Vatsal \cite{CV_DOC} which is a generalization of the work of Cornut \cite{Cornut} to CM points on Shimura curves. Taking into account remark 4.16 of \cite{CV_DUR}, we apply theorem 3.5 of \cite{CV_DOC} to the connected Shimura curve $Y_0(N)= \mathcal{H}/\Gamma_0(N)$ to obtain the following result: $\psi|_{\Gamma x}$ is surjective for all but finitely many $x \in X$ where $\Gamma x =\{ \sigma(x) \; |  \; \sigma \in \Gamma \}$. It immediately follows from this result that both $\psi|_{X^+}$ and $\psi|_{X^-}$ are surjective which as we noted above imply, as desired, that both $\dlim R_{2n} \alpha_{2n}$ and $\dlim R_{2n+1} \alpha_{2n+1}$ are not cotorsion as $\overbar{\Lambda}$-modules

\end{proof}

\begin{remark}

The main evidence in support of conjecture \ref{main_conjecture} comes from the above theorem. If either $\dlim R_{2n} \alpha_{2n}$ or $\dlim R_{2n+1} \alpha_{2n+1}$ has $\overbar{\Lambda}$-corank greater than two then the conjecture is true. Otherwise each of $\dlim R_{2n} \alpha_{2n}$ and $\dlim R_{2n+1} \alpha_{2n+1}$ has $\overbar{\Lambda}$-corank one and the conjecture in this case is equivalent to $\dlim R_{2n} \alpha_{2n} \cap \dlim R_{2n+1} \alpha_{2n+1}$ being $\overbar{\Lambda}$-torsion

Other evidence for this conjecture comes from the work of Ciperiani \cite{Ciperiani} which we will now explain. We define $R'_n=\cy{p^n}[G_n]$ and we let $R'_n \alpha_n$ be the $R'_n$-submodule of $H^1(K_n, E[p^n])$ generated by the image of $\alpha_n$ under the map

$$E(K_n) \to H^1(K_n, E[p^n])$$\\
As in section 2.2, using the natural transition maps one can show that $R'_n \alpha_n \subseteq R'_{n+2} \alpha_{n+2}$. This allows us to construct the direct limits $\dlim R'_{2n} \alpha_{2n}$ and $\dlim R'_{2n+1} \alpha_{n+2}$. Ciperiani (\cite{Ciperiani} prop. 2.1; see also \cite{CW} lemma 2.6.5) shows that the submodule of $E(K_{\infty})\otimes \Qp/\Zp$ generated by $\dlim R'_{2n} \alpha_{2n}$ and $\dlim R'_{2n+1} \alpha_{2n+1}$ has $\Lambda$-corank greater than or equal to two i.e. our conjecture is the ``mod $p$'' analog of her result.

\end{remark}

We now assume conjecture \ref{main_conjecture}. Let $M=\dlim R_{2n} \alpha_{2n} + \dlim R_{2n+1} \alpha_{2n+1}$. Since $M$ has $\overbar{\Lambda}$-corank greater or equal to two and $\overbar{\Lambda}$ is a PID therefore $M^{\dual}$ is isomorphic to $\overbar{\Lambda}^r \oplus T'$ where $r \geq 2$ and $T'$ is a finite. It follows that $M=F \oplus T$ where $F$ is cofree of rank $r$ and $T$ is finite. Since $T$ is finite there exists an $m \in \N$ such that $(g-1)^{p^m}T=\{0\}$ where $g$ is a topological generator of $\Gamma$ (note that $(g-1)^{p^m} \equiv g^{p^m}-1 \mod p$). It follows that $(g-1)^{p^m}M \subseteq F$ but as $(g-1)^{p^m}M$ has $\overbar{\Lambda}$-corank $r$ and $F$ is a cofree $\overbar{\Lambda}$-module of rank $r$, we therefore must have $F=(g-1)^{p^m}M$. Since $M$ is $\tau$-invariant therefore $(g-1)^{p^m}M$ is also $\tau$-invariant and so $F$ is $\tau$-invaraint. Composing the isomorphism $\overbar{\Lambda}^{r \, \dual} \isomarrow F$ with the inclusion $F \hookrightarrow M$ we have shown that there exists a map

$$\phi: \overbar{\Lambda}^{r \, \dual} \to \dlim R_{2n} \alpha_{2n} + \dlim R_{2n+1} \alpha_{2n+1}$$\\
such that $\img \phi$ is $\tau$-invariant and is cofree with $\overbar{\Lambda}$-corank $r \geq 2$. Therefore $\dim_{\Fp}(\img \phi)^{\Gamma} =r \geq 2$. Moreover as $\img \phi$ is $\tau$-invariant and $\tau g \tau = g^{-1}$ for any $g \in \Gamma$ therefore $(\img \phi)^{\Gamma}$ is $\tau$-invariant as well.

We are now ready to define the subgroup $S_{n_0} \subset H^(K_{n_0}, E[p])$ and the set $U$ in section 2.3. Note that the endomorphism $\tau$ of $(\img \phi)^{\Gamma}$ is diagonalizable. We consider 2 cases:\\

\noindent \textit{Case 1.} The endomorphism $\tau$ of $(\img \phi)^{\Gamma}$ has two eigenspaces\\

Choose $s,s' \in (\img \phi)^{\Gamma}$ such that $\tau s = s$ and $\tau s' = -s'$. This case is similar to case 1 in section 3. Since $s$ and $s'$ are both $\Gamma$-invariant therefore they belong to $H^1(K, E[p])$. Now let $S$ be the subgroup of $H^1(K, E[p])$ generated by $s$ and $s'$. We now let $n_0=0$ and $S_{n_0}=S$. Let $V=\Gal(L_ S/L)$ where $L=K(E[p])$. We will denote $L_{\{\Fp s\}}$ and $L_{\{\Fp s'\}}$ by $L_s$  and $L_{s'}$ respectively. As in section 3 we have

$$V=\Gal(L_s/L) \times \Gal(L_{s'}/L) =E[p] \times E[p]$$\\
Complex conjugation $\tau$ acts on $V$ by

$$\tau(x,y)\tau = (\epsilon \, \tau x, -\epsilon \, \tau y)$$\\
Let $E[p]^{\epsilon}$ denote the submodule of $E[p]$ on which $\tau$ acts as $\epsilon \in \{+1, -1 \}$. We now define a subset $U$ of $V$ as

$$U=\{(x,y) \; | \; x \in E[p]^+ -\{0\} \ \text{and} \ y \in E[p]^- - \{0\}\}$$\\
It is clear that $U^+$ generates $V^+$\\

\noindent \textit{Case 2.} The endomorphism $\tau$ of $(\img \phi)^{\Gamma}$ has one eigenspace\\

Choose 2 $\Fp$-linearly independent elements $s, s' \in (\img \phi)^{\Gamma}$. Then $\tau s = \epsilon s$ and $\tau s' = \epsilon s'$ for some $\epsilon \in \{+1, -1\}$. This case is similar to case 2 of section 3. As in that case we will consider the $\Gamma$-invariants of $\img \phi / \langle s \rangle$ and $\img \phi / \langle s' \rangle$

Since $\img \phi$ is a cofree $\overbar{\Lambda}$-module of rank $r$ therefore $\img \phi / \langle s \rangle$ is also cofree of rank $r$ and so $\dim_{\Fp}(\img \phi / \langle s \rangle)^{\Gamma}=r$. Also observe that the endomorphism $\tau$ of $(\img \phi / \langle s \rangle)^{\Gamma}$ is diagonalizable and $\dim_{\Fp}((\img \phi)^{\Gamma}/ \langle s \rangle) =r-1$. Therefore there exists an element $e \in \img \phi - (\img \phi)^{\Gamma}$ such that $e + \langle s \rangle \in (\img \phi / \langle s \rangle)^{\Gamma}$ and $\tau e +\langle s \rangle = \epsilon' e +\langle s \rangle$ for some $\epsilon' \in \{+1, -1\}$. As we showed in case 2 of section 3 we have $\epsilon' =- \epsilon$ and that we may further assume that $\tau e = -\epsilon e$. Similarly one can show that there exists an element $e' \in \img \phi - (\img \phi)^{\Gamma}$ such that $\tau e' = -\epsilon e'$ and $e' + \langle s' \rangle \in (\img \phi / \langle s' \rangle)^{\Gamma}$

Also just as we showed in section 3, we can show that $s, s' ,e$ and $e'$ are linearly independent over $\Fp$.

Now choose $n \in \N$ such that $s, s', e$ and $e'$ all belong to $H^1(K_n, E[p])$ and let $S$ be the subgroup of $H^1(K_n, E[p])$ generated by $s, s', e$ and $e'$. Then $\dim_{\Fp} S = 4 $. Note that $S$ is stable under $\Gal(K_n/\Q)$. Let $n_0=n$ and $S_{n_0} = S$. Let $V=\Gal(L_S/L)$ where $L = K_n(E[p])$. We will denote $L_{\{\Fp s\}}$ by $L_s$ and similarly for $s',e$ and $e'$. Since $\dim_{\Fp} S=4$ therefore we have

$$V=\Gal(L_s/L) \times \Gal(L_{s'}/L) \times \Gal(L_e/L) \times \Gal(L_{e'}/L) \cong E[p]^4$$\\
Complex conjugation $\tau$ acts on $V$ by

$$\tau(x,y,z,w)\tau = (\epsilon \, \tau x, \epsilon \, \tau y, -\epsilon \, \tau z, -\epsilon \, \tau w)$$\\
Let $E[p]^{\epsilon}$ denote the submodule of $E[p]$ on which $\tau$ acts as $\epsilon$. We now define

$$U_1=\{(x,0, 0, z) \; | \; x \in E[p]^{\epsilon}-\{0\} \ \text{and} \ z \in E[p]^{-\epsilon}-\{0\} \}$$\\
and
$$U_2=\{(0,y, w, 0) \; | \; y \in E[p]^{\epsilon}-\{0\} \ \text{and} \ w \in E[p]^{-\epsilon}-\{0\} \}$$\\
Finally we let

$$U=U_1 \cup U_2$$\\
It is clear that $U^+$ generates $V^+$

\begin{proposition}
For any $\ell \in \mathscr{L}(U)$ the submodule $\dlim \res_{\ell} R_{2n} \alpha_{2n} + \dlim \res_{\ell} R_{2n+1} \alpha_{2n+1}$ of $\dlim E(K_{n,\ell})/p$ has $\overbar{\Lambda}$-corank equal to 2
\end{proposition}
\begin{proof}
Consider the map we defined earlier

$$\phi: \overbar{\Lambda}^{r \, \dual} \to \dlim R_{2n} \alpha_{2n}+ \dlim R_{2n+1} \alpha_{2n+1}$$\\
For any $\ell \in \mathscr{L}(U)$ let

$$\psi_{\ell}: \dlim R_{2n} \alpha_{2n} + \dlim R_{2n+1} \alpha_{2n+1} \to \dlim \res_{\ell} R_{2n} \alpha_{2n} + \dlim \res_{\ell} R_{2n+1} \alpha_{2n+1}$$\\
be the restriction map. Then for any $\ell \in \mathscr{L}(U)$ we define $\overbar{\phi}_{\ell} =\psi_{\ell} \circ \phi$ .

To show that the $\overbar{\Lambda}$-corank of $\dlim \res_{\ell} R_{2n} \alpha_{2n} + \dlim \res_{\ell} R_{2n+1} \alpha_{2n+1}$ is equal to 2, it suffices to prove that $\img \overbar{\phi}_{\ell}$ has $\overbar{\Lambda}$-corank equal to 2, for then this implies that $\dlim \res_{\ell} R_{2n} \alpha_{2n} + \dlim \res_{\ell} R_{2n+1} \alpha_{2n+1}$ has $\overbar{\Lambda}$-corank greater or equal to 2. But by the isomorphism in property (3) of the Kolyvagin classes in section 2.2 together with proposition \ref{Iwasawa_rank_proposition} we have that the $\overbar{\Lambda}$-corank of $\dlim E(K_{n, \ell})/p$ is equal to 2 so we get equality.

We will show that $\img \overbar{\phi}_{\ell}$ has $\overbar{\Lambda}$-corank equal to 2. To show this we only need to show that $\dim_{\Fp}(\img \overbar{\phi}_{\ell})^{\Gamma}=2$. To see this we note that the map $\overbar{\phi}_{\ell}$ induces a surjection

$$\overbar{\Lambda}^{r \, \dual} \twoheadrightarrow \img \overbar{\phi}_{\ell} $$\\
Therefore $(\img \overbar{\phi}_{\ell})^{\; \dual}$ is a nonzero submodule of $\overbar{\Lambda}^r$ and hence free of rank less than or equal to $r$ over $\overbar{\Lambda}$. But as $\dim_{\Fp}(\img \overbar{\phi}_{\ell})^{\Gamma}= 2$ we therefore must have that the rank is equal to 2.

We now show that $\dim_{\Fp}(\img \overbar{\phi}_{\ell})^{\Gamma}=2$. In case 1 we have chosen $s,s' \in (\img \phi)^{\Gamma}$ such that $s$ and $s'$ belong to different eigenspaces for the action of complex conjugation $\tau$. Our definition of the set $U$ in this case gives that for any $\ell \in \mathscr{L}(U)$ we have $\res_{\ell}(s) \neq 0$ and $\res_{\ell}(s') \neq 0$. Also since $s$ and $s'$ are $\Gamma$-invariant and belong to different eigenspaces for the action of $\tau$, the same will be true for $\res_{\ell}(s)$ and $\res_{\ell}(s')$. Therefore we get the desired result in this case.

Now we consider case 2. Here we have also chosen $s,s' \in (\img \phi)^{\Gamma}$. In this case both $s$ and $s'$ belong to the same eigenspace for the action of $\tau$. We have chosen elements $e$ and $e'$ such that $e + \langle s \rangle \in (\img \phi / \langle s \rangle)^{\Gamma}$, $e' + \langle s' \rangle \in (\img \phi / \langle s' \rangle)^{\Gamma}$ and $e$ and $e'$ are eigenvectors for the action of $\tau$ belonging to a different eigenspace from $s$ and $s'$. We have in this case defined our set $U$ to be the union of 2 sets $U_1$ and $U_2$. Suppose that $\ell$ belongs to $\mathscr{L}(U_1)$. The definition of the set $U_1$ shows that $\res_{\ell}(s) \neq 0$, $\res_{\ell}(e') \neq 0$ and $\res_{\ell}(s') =0$. Since $\res_{\ell}(s') =0$ therefore $\res_{\ell}(e')$ is $\Gamma$-invariant. Then we get the desired result because $s$ and $e'$ belong to different eigenspaces for the action of $\tau$. We get a similar situation if $\ell$ belongs to $\mathscr{L}(U_2)$ thereby completing the proof.
\end{proof}

\begin{corollary}\label{Iwasawa_rank_corollary2}
For any $\ell \in \mathscr{L}(U)$ we have $\dlim H^1(K_{n, \ell}, E)[p] = \dlim \res_{\ell} R_{2n} d_{2n}(\ell) + \dlim \res_{\ell} R_{2n+1} d_{2n+1}(\ell)$
\end{corollary}
\begin{proof}
Using property (3) of the Kolyvagin classes in section 2.2, it follows from the previous proposition that $\dlim \res_{\ell} R_{2n} d_{2n}(\ell)+ \dlim \res_{\ell} R_{2n+1} d_{2n+1}(\ell)$ is a submodule of $\dlim H^1(K_{n, \ell}, E)[p]$ of $\overbar{\Lambda}$-corank equal to 2. This submodule must be equal to $\dlim H^1(K_{n, \ell}, E)[p]$ since by proposition \ref{Iwasawa_rank_proposition} $H^1(K_{n, \ell}, E)[p]$ is a cofree $\overbar{\Lambda}$-module of rank 2.
\end{proof}

We can now finally prove theorem B

\begin{theoremB}
Assume that $p$ splits in $K/\Q$, $E$ has good supersingular reduction at $p$ and conjecture \ref{main_conjecture} is true, then $\Selinf(E/K_{\infty})$ has $\Lambda$-corank equal to 2 and $X_{p^{\infty}}(E/K_{\infty})=\{0\}$
\end{theoremB}

\begin{proof}

For any $\ell \in \mathscr{L}(U)$ using proposition \ref{global_duality_proposition} together with property (2) of the Kolyvagin classes in section 2.2 it follows that $\uppsi_{\ell}(\dlim \res_{\ell} R_{2n} d_{2n}(\ell) + \dlim \res_{\ell} R_{2n+1} D_{2n+1}(\ell))=0$ and so from corollary \ref{Iwasawa_rank_corollary2} we get that $\img \uppsi_{\ell}=\{0\}$

Since according to proposition \ref{generating_Xp_proposition} $\img \uppsi_{\ell}$ with $\ell$ ranging over $\mathscr{L}(U)$ generate $X_p(E/K_{\infty})^{\; \dual}$ therefore we get that $X_p(E/K_{\infty})=\{0\}$ which implies that $X_{p^{\infty}}(E/K_{\infty})=\{0\}$ by Nakayama's lemma.

It then follows from theorem \ref{Iwasawa_rank_main_theorem} that $\corank_{\Lambda}(\Selinf(E/K_{\infty})) \le 2$

So the proof of theorem B will be complete if we can show that $\corank_{\Lambda}(\Selinf(E/K_{\infty})) \geq 2$. This follows from theorem 1.7 of \cite{Gb_LNM} (or alternatively one can use proposition 2.1 of \cite{Ciperiani}).
\end{proof}

\textbf{Acknowledgments} The author is grateful to Christophe Cornut for discussing his work with him and for suggesting using his paper with Vatsal \cite{CV_DUR} to prove theorem \ref{Heegner_modules_rank_theorem}. The research for this paper began while the author was visiting Lille1 University. The author would like to thank the Mathematics Department for their hospitality during his stay.

\end{document}